\documentclass[11pt,a4paper,twoside]{amsart}
\usepackage{amssymb,amsmath}
\usepackage{amsfonts}
\usepackage{amsthm}
\usepackage{enumerate}
\usepackage{graphicx}
\usepackage{hyperref,xcolor}

\usepackage[]{algorithm2e}

\def\NN{\mathbb{N}}
\def\RR{\mathbb{R}}
\def\XX{\mathbb{X}}
\def\JJ{\mathbb{J}}

\def\BB{L^{\infty}}
\def\MM{M^{\infty}}
\def\PP{\mathcal{P}}

\def\RWM{\text{MH}}
\def\Leb{\mu^{\text{Leb}}}

\theoremstyle{plain}
\newtheorem{thm}{Theorem}[section]
\newtheorem{lem}[thm]{Lemma}
\newtheorem{prop}[thm]{Proposition}

\theoremstyle{remark}
\newtheorem{rema}[thm]{Remark}
\theoremstyle{definition}
\newtheorem{defin}[thm]{Definition}

\title{On the Poisson equation for Metropolis-Hastings chains}

\author{Aleksandar Mijatovi\'{c}}
\address{Department of Mathematics, King's College London, UK}
\email{aleksandar.mijatovic@kcl.ac.uk}

\author{Jure Vogrinc}
\address{Department of Mathematics, Imperial College London, UK}
\email{j.vogrinc13@imperial.ac.uk}

\keywords{Poisson equation for Markov chains, variance reduction, Metropolis-Hastings algorithm, 
Central Limit Theorem, asymptotic variance, Markov Chain Monte Carlo, weak approximation}

\subjclass[2000]{60J10, 60J22}

\thanks{We thank Petros Dellaportas for suggesting the problem, Sean Meyn for helpful discussions and 
the referees for useful comments that improved the paper. AM was partially supported by the EPSRC grant EP/P003818/1.}

\begin{document}

\begin{abstract}
This paper defines an approximation scheme for a solution of the Poisson equation of 
a geometrically ergodic Metropolis-Hastings chain $\Phi$.  
The scheme is based on the idea of weak approximation and gives rise to a natural
sequence of control variates for the ergodic average
$S_k(F)=(1/k)\sum_{i=1}^{k} F(\Phi_i)$,
where 
$F$
is the force function in the Poisson equation. The main results show that the
sequence of the asymptotic variances (in the CLTs for the control-variate estimators) 
converges to zero and give a rate of this convergence. Numerical examples in the case
of a double-well potential are discussed. 
\end{abstract}

\maketitle

\section{Introduction}
 
A Central Limit Theorem (CLT) for an ergodic average
$S_k(F)=\frac{1}{k}\sum_{i=1}^{k} F(\Psi_i)$
of a Markov chain 
$(\Psi_k)_{k\in\NN}$, evolving 
according to a transition kernel $\PP$
on a general state space 
$\mathcal{X}$,
is well-known to be intimately linked with 
the solution 
$\hat F$
of the 
\textit{Poisson equation}
\begin{equation}\label{PE}
\hat{F}-\PP\hat{F}= F-\pi( F) 
\tag{PE($\PP$, $ F$)}
\end{equation}
with a \emph{force function} 
$F\colon \mathcal{X}\to\RR$
(see~\cite[Sec.17.4]{tweedie}). 
Here $\pi$ is the invariant probability measure of $\Psi$ on $\mathcal{X}$,
$\pi(F)=\int_{\mathcal{X}} F(x)\pi(dx)$ 
and
$\PP  G(x)=\mathbb{E}_x[G(\Psi_{1})]$ 
for any $G:\mathcal{X}\to\RR$.
In fact, the Poisson equation in~\eqref{PE}
is of fundamental importance in many areas of probability, statistics and
engineering 
(see~\cite[Sec.17.7, p.459]{tweedie}).
In this context one of the main motivations for constructing approximations to $\hat F$ is to reduce 
the asymptotic variance in~\eqref{CLT} for the Markov Chain Monte Carlo (MCMC) estimators, thus speeding up the MCMC algorithms.

Assume that the random sequence 
$(S_k(F))_{k\in\NN}$
satisfies the strong law of large numbers (SLLN),
$\lim_{k\to\infty} S_k(F)=\pi(F)$ a.s., 
and the CLT
\begin{equation}\label{CLT}
\sqrt{k}\left(S_k(F)-\pi( F)\right)\overset{d}{\longrightarrow}\sigma_ F\cdot N(0,1)\qquad \text{as $k\to\infty$,}
\tag{CLT($\Psi$, $ F$)}
\end{equation}
where $N(0,1)$
is a standard normal distribution and the constant 
$\sigma_ F^2$
is the \emph{asymptotic variance}.
Put differently, the variance of $S_k(F)$ is 
approximately equal to $\sigma_F^2/k$. 
It is hence intuitively clear that if $\sigma_F^2$ is large, which occurs in applications particularly when
$F$ has super-linear growth (as $\sigma_F^2\propto\mathrm{Var}_\pi(F)$, 
see e.g.~\cite[Sec.5]{roberts} and the references therein), the variance of 
$S_k(F)$ will also be big, requiring a large number of steps $k$ for convergence.  
In contrast, 
imagine we knew the solution $\hat F$ of the Poisson equation~\eqref{PE} 
and could evaluate the function 
$\PP\hat{F}-\hat{F}$.
Then the estimator given by the ergodic average 
$S_k(F+\PP\hat{F}-\hat{F})$ (for any $k\in\NN$) would 
be equal to the constant $\pi(F)$ for any (not necessarily stationary) path of
the chain $\Psi$, i.e.
its variance vanishes for $\pi$-a.e. starting point. 
However, solving Poisson's equation for the chains arising in most applications, even
for very simple functions $F$, is for all practical purposes impossible (see
e.g. relevant comments in~\cite{henderson}). Nevertheless, this line of reasoning suggests the
following heuristic: 

\smallskip

\noindent \emph{a good approximation $\tilde F$
to a solution of~\eqref{PE} significantly reduces the asymptotic variance in
the}~\hyperref[CLT]{(CLT($\Psi$, $ F+\PP \tilde F-\tilde F$))}, \emph{i.e.} 
$\sigma^2_{F+\PP \tilde F-\tilde F}\ll\sigma_ F^2$.

\smallskip


This heuristic is well known and strongly substantiated with numerical evidence. As a method of variance reduction it has been developed in various
Markovian settings~\cite{andradottir, henderson, HendersonGlynn02, MeynTadic}. 
Its applications in stochastic networks theory
are described in~\cite[Ch.~11]{Meyn_Control}, 
while applications in statistics 
for the random scan Gibbs sampler 
were developed in~\cite{petros}. 
However,
schemes for constructing $\tilde{F}$ found in the literature (a)~depend strongly on the structure of the underlying model and,
to the best of our knowledge,
(b)~there are no theoretical results quantifying~\textit{a~priori} the amount of reduction in the asymptotic variance 
of~\hyperref[CLT]{CLT($\Psi$, $ F+\PP \tilde F-\tilde F$)}. 
This paper aims to address both (a) and (b) by introducing a general~\hyperref[alg1]{\textbf{Scheme}} (see below) 
for constructing an approximate solution $\tilde{F}$
to~\eqref{PE}, applicable to any discrete time Markov chain,
and analysing the corresponding asymptotic variance in the setting of Metropolis-Hastings chains.

\smallskip

\begin{algorithm}[H]
\begin{center}
\textbf{Scheme}
\end{center}
 \KwIn{Transition kernel $\PP$, function $F$, 
allotment $\XX=(\JJ,X)$ 
consisting of a partition $\JJ=\{J_0,J_1,\dots,J_m\}$ of $\mathcal{X}$ and representatives 
$X=\{a_j\in J_j:j=0,1,\dots,m\}$.}
\Begin{
(I) Define $p_{\XX}\in\RR^{(m+1)\times (m+1)}$ 
and $f_{\XX}\colon\{a_0,a_1,\dots a_m\}\to \RR$ 
respectively  by
$$(p_{\XX})_{ij}:=\PP(a_i,J_j)\quad\text{and} \quad
f_{\XX}(a_j):=F(a_j),\quad \text{ where $i,j\in\{0,1,\dots,m\}$.}$$
(II) Find a solution $\hat f_{\XX}$ of Poisson's equation~\hyperref[PE]{(PE($p_{\XX}$,$f_{\XX}$))}. \\
(III) Define $\tilde{F}_{\XX} :=\sum_{j=1}^m  \hat{f}_{\XX}(a_j)1_{J_j}$.}
\KwOut{Approximate solution $\tilde{F}_{\XX}:\mathcal{X}\to\RR$ to Poisson's equation in~\eqref{PE}.} 
\label{alg1}
\end{algorithm}

\smallskip

Our main result (Theorem~\ref{thetheorem} below)
states that,  
for an appropriately 
chosen allotment $\XX$, 
the function $\tilde{F}_{\XX}$ 
can theoretically achieve an arbitrary
reduction of the asymptotic variance for a class of Metropolis-Hastings
chains and force functions $F$ that satisfy natural growth conditions. 
To the best of our knowledge, this is the first systematic approach capable of reducing the asymptotic
variance arbitrarily for a general class of discrete-time Markov chains. 
The proof hinges on the uniform convergence to stationarity of a sequence of approximating 
Markov chains, which in turn crucially depends on the results in~\cite{meyn,baxendale} (see Section~\ref{subsec:Overview} below for details).
Step (II) in the \hyperref[alg1]{Scheme} amounts to solving a linear system and
can be carried out provided that the stochastic matrix $p_{\XX}$ is
irreducible. Moreover, Poisson's
equation~\hyperref[PE]{(PE($p_{\XX}$,$f_{\XX}$))} has a solution that is unique
up to the addition of a constant function (see \cite[Theorem~9.3]{makowski}).
Furthermore, the asymptotic variance in~\hyperref[CLT]{CLT($\Psi$, $F+\PP \tilde{F}_{\XX}-\tilde F_{\XX}$)} 
does not depend on the choice of $\hat{f}_{\XX}$ in step (II) of the \hyperref[alg1]{\textbf{Scheme}}.

The approximation~\hyperref[alg1]{\textbf{Scheme}} exploits the stochastic evolution implicitly present 
in~\eqref{PE}. As in~\cite{HendersonGlynn02, MeynTadic,Meyn_Control}, 
we are using the true solution of the Poisson equation for a 
related Markov process to construct $\tilde{F}$.
In our context, the approximation of $\hat F$ is based on 
the \textit{weak approximation} of the chain $\Psi$
by a sequence of ``simpler'' finite state Markov chains (converging in law to $\Psi$), 
such that the solutions of the Poisson equations for the approximating
chains can be characterised algebraically. The approximating Markov chain underpinning the~\hyperref[alg1]{\textbf{Scheme}}  
mimics the  behaviour of $\Psi$ as follows: its state space is a partition
$\{J_0,J_1,\dots,J_m\}$ of the state space $\mathcal{X}$
and its transition matrix consists of the 
probabilities of $\Psi$ jumping from a chosen element in $J_i$
into the set $J_j$. Analogous weak approximation ideas
have been applied in continuous time to Brownian
motion~\cite{M07}, L\'evy~\cite{MVJ} and Feller~\cite{MP} processes.
A recent interesting approach for approximating the solution of Poisson's equation in discrete 
time has been proposed in~\cite{Poisson_App_Derivative}. The idea is to solve the
equation obtained by differentiating both sides of~\eqref{PE} in the state variable. This leads to 
a new approximation method for $\hat F$ but appears to
require smoothness properties of the transition kernel, not afforded by the class of 
Metropolis-Hastings chains.  

The approximation of a given Markov chain with a finite-state chain given by the~\hyperref[alg1]{\textbf{Scheme}} is akin to others previously mentioned in the literature that are also based on a partition or a covering of the state space, see for instance \cite{runnenburg62,hoekstra84,rosenthal92} and \cite{MadrasRandall02}. These papers relate the speed of convergence to equilibrium of the initial and of the approximating Markov chains. They do not however address potential similarity of Poisson's equations.

Theorem~\ref{thetheorem} is theoretical in nature as the partition in $\XX$ that provably reduces the variance below a prescribed level 
typically requires a large number of approximating states $m$.  However, Example~\ref{subsec:Example2} in Section~\ref{subsec:Example} below demonstrates numerically
that in the case of a Random walk Metropolis chain converging to a double-well potential,
the~\hyperref[alg1]{\textbf{Scheme}} applied with only $m=6$ points 
reduces the variance by approximately 10\%
(see Section~\ref{section6} below for details).

A natural question arising from Theorem~\ref{thetheorem} is about the rate of the decay
of the sequence of asymptotic variances  $\sigma^2_n\to0$. Theorem~\ref{rate}
shows that
the decay is governed by the greater of the two quantities: the mesh of the partition 
of the bounded set $\RR^d\setminus J^n_0$ and the $\pi$-average 
of the squared drift function of the chain over $J^n_0$
(see Section~\ref{section2} for  definitions).
Furthermore, for the chains studied in~\cite{roberts2, jarner},
Theorem~\ref{rate} implies a bound on the rate of decay in terms of the target
density $\pi$ alone
(see Proposition~\ref{vrana} below). 
We hope this result is both of some practical value 
(cf. Section~\ref{subsec:Example1}) and independent interest.

The reminder of the paper is organised as follows. Section~\ref{section2} 
formulates our main result (Theorem~\ref{thetheorem}). 
In Section~\ref{section5} we prove Theorem~\ref{thetheorem}. 
The structure of the proof is given in Section~\ref{subsec:Overview},
while Sections~\ref{subsec:Bounding},~\ref{subsec:Uniform} and~\ref{subsec:ApproximateSolution} carry out the steps.
In Section~\ref{sec:Rate} we state and prove Theorem~\ref{rate} and 
Proposition~\ref{vrana}, bounding the rate of convergence to zero of the asymptotic variances. Section~\ref{section6}
describes the implementation of the~\hyperref[alg1]{\textbf{Scheme}}  (Section~\ref{subsec:implementation}) and gives numerical examples
(Section~\ref{subsec:Example}).

\section{Assumptions and the main result}\label{section2}

Let $\pi$ be a density function of a probability measure on $\RR^d$ with respect to the Lebesgue measure $\Leb$ 
and let $q\colon \RR^d\times\RR^d\to\RR$ be a transition density function, i.e. for every
$x\in\RR^d$, the function $y\mapsto q(x,y)$ is a density on $\RR^d$. 
The idea behind the dynamics of a Metropolis-Hastings chain is to propose a move 
from a density $q(x,\cdot)$ to a new location, say $y$, and accept it 
with probability 
$$
\alpha(x,y):=
\begin{cases}
\min\left(1,\frac{\pi(y)q(y,x)}{\pi(x)q(x,y)}\right),& \pi(x)q(x,y)> 0,\\
1,& \pi(x)q(x,y)=0.
\end{cases}
$$
The Markov transition kernel $P(x,dy)$ for this dynamics is given by the formula 
\begin{equation}\label{kernelform}
P(x,dy):= 
\alpha(x,y)q(x,y)dy+\left(1-\int_{\RR^d}{\alpha(x,z)q(x,z)dz}\right)\delta_x(dy),
\tag{MH($q$, $\pi$)}
\end{equation}
where $\delta_x$ is Dirac's measure centred at $x$,
and the Markov chain $(\Phi_k)_{k\in\NN}$ generated by $P$ is known as the \emph{Metropolis-Hastings chain} (see~\cite{metropolis,hastings}).
In this context, $\pi$ is termed a \emph{target density} and $q$ a \emph{proposal density}.
It is easy to see that 
the chain $\Phi$ is reversible (i.e. it satisfies $\pi(x)dxP(x,dy)=\pi(y)dy P(y,dx)$) and hence 
stationary
(i.e. $\int_{\RR^d} P(x,dy)\pi(x)dx=\pi(y)dy$) with respect to
$\pi$.
The measure $\pi(x)dx$ is also known as the \textit{invariant probability measure} for the chain $\Phi$. 
Throughout the paper we assume that the kernel $P$ in~\hyperref[kernelform]{$\RWM(q,\pi)$} satisfies the following assumptions:
\begin{enumerate}
\item[A1:]\label{A1} \textit{Geometric drift condition:} 
there exists a continuous  fucntion
$V\colon\RR^d\to [1,\infty)$, such that
$\pi(V^2)<\infty$, $V$ has bounded sublevel sets (i.e. 
$V^{-1}\left(\left[1,c\right] \right)$ is bounded for every $c\geq1$) 
and
\begin{equation}
P V(x)\leq \lambda_V V(x)+\kappa_V 1_{C_V}(x)\text{,}\hspace{10pt}\text{for all }x\in\RR^d\text{,}\nonumber
\end{equation}
for constants 
$\lambda_V\in(0,1)$, $\kappa_V>0$ 
and a compact set
$C_V\subset \RR^d$. 
\item[A2:] \label{A2} The target density $\pi\colon\RR^d\to(0,\infty)$ is continuous and strictly positive.
\item[A3:] \label{A3} The proposal density $q\colon \RR^d\times\RR^d\to(0,\infty)$ is continuous, strictly positive and  bounded.
\end{enumerate}

Associated with the \emph{drift function} $V$ is the function space
\begin{equation}\label{BV}
\BB_V:=\left\{G\colon\RR^d\to\RR\text{; $G$ measurable and } ||G||_V<\infty\right\},\quad\text{where}\quad
||G||_V:=\sup_{x\in\RR^d}\frac{|G(x)|}{V(x)}.
\end{equation}
Note that 
$\BB_V$
equipped with the norm
$||\cdot||_V$ is a Banach space (see~\cite[Proposition~7.2.1]{lasserre}). 

\begin{rema}
\label{A3remark}
\begin{enumerate}[(i)]
\item Assumptions \hyperref[A1]{A1}-\hyperref[A3]{A3} are standard. 
Widely used classes of 
Random walk Metropolis chains (i.e. $q(x,y)=q^*(y-x)$) 
satisfying~\hyperref[A1]{A1}-\hyperref[A3]{A3} 
are given in~\cite{mengersen, roberts2, jarner}. 
See also~\cite{robertsmala} for examples of Metropolis Adjusted Langevin
chains satisfying~\hyperref[A1]{A1}-\hyperref[A3]{A3}. 

\item For Metropolis kernel $P$ satisfying \hyperref[A1]{A1}-\hyperref[A3]{A3}
and $F\in\BB_V$ there exists a solution $\hat{F}$ to \hyperref[PE]{PE($P$, $ F$)} 
that is an element of $\BB_V$. The solution $\hat{F}$ is unique up to the addition
of a constant function (see \cite[Prop~1.1 and Thm~2.3]{lyapunov}).

\item Assumptions \hyperref[A2]{A2} and \hyperref[A3]{A3} imply that
Metropolis-Hastings chain $\Phi$ driven by
$P$ is $\pi$-irreducible (i.e $\Leb$-irreducible), 
strongly aperiodic and positive Harris recurrent (see~\cite[Lem~1.1\&1.2]{mengersen},~\cite[Thm~1, Cor~2]{tierney}  
and monograph~\cite{tweedie} as a general reference).
In particular,  
the SLLN~\cite[Thm~17.1.7]{tweedie}
and the CLT~\cite[Thm~17.4.4]{tweedie} hold for $F\in\BB_V$.  

\item If $\pi(V)<\infty$ but $\pi(V^2)=\infty$, we may work with $\sqrt{V}$ instead of $V$,  
as Jensen's inequality implies $P(\sqrt{V})\leq \sqrt{\lambda_V} \sqrt{V}+\sqrt{\kappa_V}1_{C_V}$,
thus restricting our results to force functions $F\in\BB_{\sqrt{V}}$.

\item Geometric drift condition~\hyperref[A1]{A1} implies that 
for $G\in\BB_V$ we have
$\pi(G^2)<\infty$,
$P G(x)$ 
is well defined for any $x\in\RR^d$,
$PG\in\BB_V$  and $\pi(P G-G)=0$.  
In particular, \hyperref[CLT]{CLT($\Phi$, $F+PG-G$)} holds with mean $\pi(F)$ 
and (possibly substantially reduced) asymptotic variance $\sigma_{F+P G-G}^2$. 
\end{enumerate}
\end{rema}

Remark~\ref{A3remark}(v) motivates the following definition. 

\begin{defin}\label{approx} Let $\Phi$ be a Metropolis-Hastings chain driven by
kernel $P$ in~\hyperref[kernelform]{$\RWM(q,\pi)$}. 
Let $(G_n)_{n\in\NN}$ be a sequence in $\BB_V$
with the asymptotic variance $\sigma^2_n$
in the~\hyperref[CLT]{CLT($\Phi$, $ F+ P G_n-G_n)$}. 
We say that $(G_n)_{n\in\NN}$ \emph{asymptotically solves Poisson's equation} \hyperref[PE]{PE($P$, $ F$)} if
$\lim_{n\to\infty} \sigma^2_n=0$.
\end{defin}

\begin{rema}
\begin{enumerate}[(a)]
\item If $(G_n)_{n\in\NN}$ asymptotically solves Poisson's equation \hyperref[PE]{PE($P$,$F$)},
so does $(G_n+c_n)_{n\in\NN}$ for any sequence $(c_n)_{n\in\NN}$ of real numbers. 
\item Definition~\ref{approx} does not require the Metropolis-Hastings structure on $\RR^d$ and 
can be extended trivially to Markov chains on general state spaces satisfying an appropriate CLT. 
\end{enumerate}
\end{rema}

We now define
a sequence of functions that asymptotically solves Poisson's equation \hyperref[PE]{PE($P$, $ F$)}.
\begin{defin}
\label{allotmentdef}
\begin{itemize}
\item[\textbf{(a)}] 
Let $\mathbb{J}$ be a partition of $\RR^d$ into measurable sets
$J_0,J_1,\dots,J_m$, such that $\cup_{j=1}^m J_j$ is bounded and $\Leb(J_j)>0$
holds for all $0\leq j\leq m$. Let $X=\{a_0,a_1,\dots,a_m\}$ be a set of
\emph{representatives}: $a_j\in J_j$ for all $0\leq j\leq m$.  
The pair $\XX:=(\mathbb{J},X)$ 
is called an \emph{allotment} and $m$ be the \emph{size} of the allotment
$\XX$.

\item[\textbf{(b)}]
Let $W\colon \RR^d\to [1,\infty)$ be a measurable function and $\XX$ an allotment. \emph{$W$-radius} and \emph{$W$-mesh} of 
the allotment $\XX$ are defined by 
\begin{eqnarray}\label{radiusdef}
\text{rad}(\XX,W) & := & \inf_{y\in J_0} W(y),\\
\delta(\XX,W) & := & \max\left(\max_{1\leq j\leq m}\sup_{y\in J_j}|y-a_j|,\max_{0\leq j\leq m}\sup_{y\in J_j}
(W(a_j)/W(y)-1)\right),
\label{meshdef}
\end{eqnarray}
respectively,
where $|x|$ denotes the Euclidean norm of any $x\in\RR^d$.

\item[\textbf{(c)}]
A sequence of allotments
$\left(\XX_{n}\right)_{n\in\NN}$ is \emph{exhaustive} with respect to the function $W$ 
in \textbf{(b)}
if the following holds:
$\lim_{n\to\infty}\text{rad}(\XX_{n},W)=\infty$ and
$\lim_{n\to\infty}\delta(\XX_{n},W)= 0$.
\end{itemize}
\end{defin}

\begin{rema}
\begin{enumerate}[(i)]
\item For any continuous function $W\colon \RR^d\to [1,\infty)$ with
bounded sublevel sets, there exists an exhaustive sequence of allotments (see Appendix~\ref{appendix} below).

\item
Note that $J_0$ is the only unbounded set in the partition of an allotment $\XX$. 
For the $W$-radius of
$\XX$
to be large, the union $\cup_{j=1}^{m}J_j$ of all the bounded sets in the partition
has to cover the part of $\RR^d$ where $W$ is small.

\item
The $W$-mesh is a maximum of two quantities:
the first is a standard mesh of the partition 
$\{J_1,\dots,J_m\}$
of the bounded set
$\RR\setminus J_0=\cup_{j=1}^m J_j$. The
second quantity in~\eqref{meshdef}
implies that for the $W$-mesh to be small, representatives
$a_j$ have to be chosen so that $W(a_j)$ and $\inf_{y\in J_j}W(y)$ are close
to each other, relative to size of $W$ on $J_j$.  
Intuitively, if $W(a_0)$ is close to $\inf_{y\in J_0}W(y)$ 
and $W$ is continuously differentiable, 
then the second
term in~\eqref{meshdef} is approximately equal to
$$\max_{1\leq j\leq m}\sup_{y\in J_j}\left((\nabla\log W (y))^{\top}(y-a_j)\right)\text{.}$$
Thus, if $W$ does not exhibit super-exponential growth, the representatives $a_1,\ldots,a_m$  can be chosen arbitrarily.
\end{enumerate}
\end{rema}

We can now state our main result.

\begin{thm}\label{thetheorem}
Let the transition kernel $P$ in~\hyperref[kernelform]{$\RWM(q,\pi)$} of a Metropolis-Hastings chain $\Phi$ satisfy
\hyperref[A1]{A1}-\hyperref[A3]{A3} for a drift function $V$. Let $F\in\BB_V$ be continuous $\pi$-a.e. and let
$\left(\XX_n=(\mathbb{J}_n,X_n)\right)_{n\in\NN}$ be an exhaustive sequence of allotments with
respect to $V$, where $\JJ_n=\{J_0^n,\ldots,J_{m_n}^n\}$ and $X_n=\{a^n_j\in J^n_j:j=0,1,\dots,m_n\}$. 
For each $n\in\NN$, let $\tilde{F}_n$ be the output of the~\hyperref[alg1]{\textbf{Scheme}} 
with input $P$, $F$ and $\XX_n$. Then 
the sequence $(\tilde{F}_n)_{n\in\NN}$ asymptotically solves Poisson's
equation~\hyperref[PE]{PE($P$,$F$)}, i.e.
the asymptotic variance $\sigma^2_n$ in~\hyperref[CLT]{CLT($\Phi$, $ F+P\tilde{F}_n-\tilde{F}_n$)}
converges to zero as $n\to\infty$. 
\end{thm}

\begin{rema}\label{pXremark}
Functions $\tilde{F}_n$ from Theorem~\ref{thetheorem} are well defined. This is because all the entries
\begin{equation}\label{pX}
{(p_n)_{ij}:=(p_{\XX_n})_{ij}}=P(a_i^n,J_j^n)=\begin{cases} \int_{J_j^n}\alpha(a_i^n,y)q(a_i^n,y)dy &\text{if}\hspace{10pt} i\neq j \\
1-\int_{\RR^d\setminus J_i^n}\alpha(a_i^n,y)q(a_i^n,y)dy &\text{if}\hspace{10pt} i= j
\end{cases}
\end{equation}
of stochastic matrices $p_n$, constructed by the \hyperref[alg1]{\textbf{Scheme}} with input $P$, $F$ and $\XX_n$, 
are strictly positive by assumptions~\hyperref[A2]{A2},~\hyperref[A3]{A3} 
and Definition~\ref{allotmentdef}\textbf{(a)}
($\Leb(J_j^n)>0$ for all $0\leq j\leq m_n$, where $m_n$ is the size of allotment $\XX_n$). 
Hence the chain on $X_n$, driven by $p_n$, is irreducible, recurrent, aperiodic and admits a
unique invariant probability measure $\pi_n$. Moreover, Poisson's equation for $p_n$ and
any force function on $X_n$ has a solution, unique up to addition
of a constant (see \cite[Theorem~9.3]{makowski}).
\end{rema}

\begin{rema}\label{rem:generalising to non MH kernels}
The proof of Theorem~\ref{thetheorem} does not rely heavily on the structure of Metropolis-Hastings kernels. Emulating the proof appears feasible at least for other specific $T$-chains (see \cite[Chapte~6]{tweedie} for the definition). More specifically, reversibility is needed in Proposition~\ref{l2var}, but an analogous result can be obtained without it. In the proof of Proposition~\ref{uniformdrift} b), we use the fact that the non-Dirac component $T(x,dy)$ of $P(x,dy)$ has positive and continuous density with respect to the Lebesgue measure. Finally, in proofs of Proposition~\ref{kje2} and Theorem~\ref{thetheorem} we require $T(x,dy)$ to exhibit the following form of continuity, $\lim_{n\to\infty}\left\|T(x,\cdot)-T(a^n(x),\cdot)\right\|_V=0$ for $\pi$-a.e. $x$ (here $\|\cdot\|_V$ is the $V$ total variation norm and $a^n(x)=\sum_{j=0}^{m_n}a^n_j1_{J^n_j}(x)$).
\end{rema}

`\section{Proof of Theorem~\ref{thetheorem}}\label{section5}

\subsection{Overview of the proof}
\label{subsec:Overview}

The central object in the proof of Theorem~\ref{thetheorem} is the function
\begin{equation}\label{eq:Delta_n}
\Delta(G):=PG-G+F-\pi(F),
\end{equation}
which measures the failure of a function $G$ to be a solution of the Poisson
equation~\hyperref[PE]{PE($P$,$F$)}. Intuitively, the closer $\Delta(G)$ is to zero the better. 

The proof is in two parts. In the first part (Section~\ref{subsec:Bounding} below)
we show that a sequence of functions $(G_n)_{n\in\NN}$ in $\BB_V$
asymptotically solves Poisson's equation~\hyperref[PE]{PE($P$,$F$)}
if
$\lim_{n\to\infty}\pi(\Delta(G_n)^2)= 0$. 
This is a simple consequence 
of the representation of the asymptotic variance in terms of the spectral
measure~\cite[Eq~1.1]{KipnisVaradhan86} and the existence of a spectral gap
for geometrically ergodic Markov chains established~\cite[Prop~1.1]{RobertsRosenthal97}.


The second part of the proof is more involved. 
It consists of verifying that functions $(\tilde{F}_n)_{n\in\NN}$, 
defined in Theorem~\ref{thetheorem}, indeed satisfy $\lim_{n\to\infty}\pi(\Delta(\tilde{F}_n)^2)=0$.
The key underlying fact needed for this purpose
is that the family of the approximating finite state Markov chains driven by the stochastic matrices
$(p_n)_{n\in\NN}$ 
converge to their respective stationary distributions $(\pi_n)_{n\in\NN}$
\textbf{uniformly} in $n\in\NN$. This step is facilitated by the results 
in~\cite[Thm~2.3]{meyn} and~\cite[Thm~1.1]{baxendale},
which show that the constants appearing in the geometric ergodicity estimate 
depend only and explicitly on the constants in the drift, minorisation and strong aperiodicity conditions for that chain. 
In Section~\ref{subsec:Uniform} we show that 
these constants can be chosen independently of 
$n\in\NN$ 
(Proposition~\ref{uniformdrift} below) 
and establish the uniform convergence to stationarity
(Proposition~\ref{operatorprop} below). 

In Section~\ref{subsec:ApproximateSolution} 
we establish convergence in $L^2(\pi)$  of the sequence
$(\Delta(\tilde{F}_n))_{n\in\NN}$.
In addition to the uniform convergence to stationarity,
the proof requires a further weak approximation by a family of finite state Markov chains
with stationary distributions that are explicit in the target density $\pi$
(see~\eqref{eq:*_chain} below).
Note that  the stationary laws $\pi_n$ of the chains generated by the stochastic matrices $p_n$, defined in~\eqref{pX}, 
cannot be expressed explicitly in terms of $\pi$.

\begin{rema}\label{rem:notation} \textbf{Auxiliary notation:}
In addition to the notation used in the statement of Theorem~\ref{thetheorem}
and Remark~\ref{pXremark}, throughout the remainder of the section we will use the following objects: 

\begin{itemize}
\item $\hat{F}$: solution of~\hyperref[PE]{PE($P$,$F$)} in $\BB_V$ (\textit{cf}. Remark~\ref{A3remark}(ii)).
\item $f_n$ and $v_n$: restrictions of $F$ and $V$ to the set $X_n$, respectively. 
\item $\hat{f}_n$: solution of~\hyperref[PE]{PE($p_n$,$f_n$)} constructed within the~\hyperref[alg1]{\textbf{Scheme}} (\textit{cf}. Remark~\ref{pXremark}).
\item $\delta_n:=\delta(\XX_n,V)$: the $V$-mesh of the allotment $\XX_n$ defined in~\eqref{meshdef}.
\end{itemize}
\end{rema}

\subsection{Controlling the asymptotic variance} 
\label{subsec:Bounding}
The following proposition gives a sufficient conditions for a sequence of functions
$(G_n)_{n\in\NN}$ to solve asymptotically the Poisson equation.

\begin{prop}\label{l2var}
Let  the sequence $(G_n)_{n\in\NN}$
in $\BB_V$ satisfy $\lim_{n\to\infty}\pi\left(\Delta(G_n)^2\right)=0$.	
Then $(G_n)_{n\in\NN}$ asymptotically solves~\hyperref[PE]{PE($P$, $F$)} in the sense of Definition~\ref{approx}.
\end{prop}

\begin{proof}
The kernel $P$ is reversible and hence a bounded self-adjoint operator on the Hilbert space $L^2(\pi)$.
Furthermore, the Hilbert subspace $\mathcal{H}:=\{G\in L^2(\pi):\pi(G)=0\}$ is invariant for $P$ (i.e.
$\pi(PG)=0$ for any $G\in\mathcal{H}$). 
By~\eqref{eq:Delta_n} and Remark~\ref{A3remark}(v) it follows that  
$\Delta(G_n)\in\mathcal{H}$ for all $n\in\NN$. 
The asymptotic variance $\sigma^2_n$ in
the~\hyperref[CLT]{CLT($\Phi$, $F+P G_n- G_n)$} can be represented in terms of
a positive (spectral)  measure $E_{\Delta(G_n)}(d\lambda)$ on the spectrum $\sigma(P|_\mathcal{H})\subset\RR$ associated with the function
$\Delta(G_n)$, as follows 
(see \cite{KipnisVaradhan86} and \cite[Thm~2.1]{Geyer92} for details):
\begin{equation}
\label{eq:spect_rep}
\sigma^2_n=\int_{\sigma(P|_\mathcal{H})}\frac{1+\lambda}{1-\lambda}E_{\Delta(G_n)}(d\lambda)\text{.}
\end{equation}
Since the chain generated by $P$ is geometrically ergodic by~\hyperref[A1]{A1},~\cite[Prop~1]{RobertsRosenthal97}
implies that the spectral radius $\rho$ of 
$P|_\mathcal{H}$ 
satisfies $\rho<1$.
Hence the inclusion 
$\sigma(P|_\mathcal{H})\subseteq[-\rho,\rho]$, the equality 
$E_{\Delta(G_n)}(\sigma(P|_\mathcal{H}))=\pi\left(\Delta(G_n)^2\right)$
(see e.g.~\cite[Eq~(2.2)]{Geyer92})
and the formula in~\eqref{eq:spect_rep} imply
$$\sigma^2_n\leq \frac{1+\rho}{1-\rho}\cdot\int_{\sigma(P|_\mathcal{H})}E_{\Delta(G_n)}(d\lambda)=\frac{1+\rho}{1-\rho}\cdot\pi\left(\Delta(G_n)^2\right)\longrightarrow 0\qquad
\text{ as $n\to\infty$.}$$
This proves the proposition.
\end{proof}

\subsection{Uniform convergence to stationarity}
\label{subsec:Uniform}
Fix an exhaustive sequence of
allotments  $(\XX_n)_{n\in\NN}$
and stochastic matrices $p_n$, $n\in\NN$, as in Theorem~\ref{thetheorem}.
The main aim of this section is to 
prove that the corresponding chains are geometrically ergodic uniformly in $n\in\NN$.
This is achieved as follows:  first, 
the \emph{uniform drift, minorisation} and \textit{strong aperiodicity conditions} in~\eqref{Dn},~\eqref{uniformminor} and~\eqref{gammakaca}, respectively,
are established. 
Then, the uniform convergence to stationarity follows from~\cite[Thm~1.1]{baxendale} (cf.~\cite[Thm~2.3]{meyn}). 

For each $n\in\NN$, let $a^n\colon \RR^d\to\RR^d$ 
map $x\in\RR^d$ to its representative in $\XX_n$. 
More precisely, let
\begin{equation}\label{aodx}
a^n(x):=\sum_{j=0}^{m_n} a_j^n 1_{J_j^n}(x)\qquad\text{for every $x\in\RR^d$,}
\end{equation}
where 
$\{J_0^n,\ldots,J_{m_n}^n\}$ is the partition 
and 
$X_n=\{a_0^n,\ldots,a_{m_n}^n\}$
are the representatives
in the allotment $\XX_n$.
Since the sequence of allotments is exhaustive, the following limit holds: 
\begin{equation}\label{aconv}
\lim_{n\to\infty} a^n(x)=x\qquad\text{for every $x\in\RR^d$.}
\end{equation}
Note that the definition of a $V$-mesh (see~\eqref{meshdef} in Definition~\ref{allotmentdef}) implies 
the inequality
\begin{equation}\label{Vbound}
V(a^n(x))=V(a^n(x))-V(x)+V(x)\leq (1+\delta_n)V(x)\qquad \text{for all $n\in\NN$ and $x\in\RR^d$.}
\end{equation}

\begin{prop}\label{uniformdrift}{\textbf{Uniform drift, minorisation and strong aperiodicity conditions.}}
There exists a compact set $C\subset \RR^d$ such that the following statements hold.
\begin{itemize}
\item[\textbf{(a)}]
There exist positive constants $\lambda<1,\kappa$, such that
the uniform drift condition holds:
\begin{equation}\label{Dn}
p_nv_n(a^n_j)\leq \lambda v_n(a^n_j)+\kappa 1_{C}(a^n_j)\hspace{10pt}\text{for all $n\in\NN$, and $a^n_j\in X_n$.}
\end{equation}
\item[\textbf{(b)}]
Define $C_n:= X_n\cap C$,
for each $n\in\NN$. There exist constants
$\gamma,\tilde{\gamma}\in(0,\infty)$ and a measure $\nu_n$, concentrated on $X_n$,
such that the uniform minorisation condition, 
\begin{equation}\label{uniformminor}
(p_n)_{ij}\geq \gamma \nu_n\big(\{a_j^n\}\big)\hspace{10pt}\text{for all $n\in\NN$, and $i,j\in\{0,1,\dots,m_n\}$ satisfying $a^n_i\in C_n$,}
\end{equation}
and the uniform strong aperiodicity condition, 
\begin{equation}\label{gammakaca}
\gamma \nu_n(C_n)\geq \tilde{\gamma}\hspace{10pt}\text{for all $n\in\NN$,}
\end{equation}
hold.
\end{itemize}
\end{prop}

\begin{proof} 
\textbf{(a)} Fix an arbitrary $n\in\NN$ and $j\in\{0,\ldots, m_n\}$. 
By definition of the function $a^n(\cdot)$ in~\eqref{aodx}, we find
$$p_nv_n(a^n_j)-v_n(a^n_j)=\int_{\RR^d} \big(V(a^n(y))-V(a_j^n)\big)\alpha(a^n_j,y)q(a_j^n,y)dy\text{.}$$
By~\eqref{Vbound} we get
$V(a^n(y))-V(a^n_j)\leq V(y)-V(a^n_j)+\delta_nV(y)$ for every $y\in\RR^d$. The form of kernel $P$ 
in~\eqref{kernelform}
and this inequality imply
\begin{align*}
p_nv_n(a^n_j)-v_n(a^n_j)&\leq P V(a^n_j)-V(a^n_j)+\delta_n\int_{\RR^d}V(y)\alpha(a^n_j,y)q(a_j^n,y)dy\\
&\leq  P V(a^n_j)-V(a^n_j)+\delta_n P V(a^n_j)=(1+\delta_n)P V(a^n_j)-V(a^n_j).
\end{align*}
Since by definition $V(a^n_j)=v_n(a^n_j)$, the geometric drift condition in~\hyperref[A1]{A1} implies 
$$p_nv_n(a^n_j)\leq  (1+\delta_n)\lambda_V v_n(a^n_j)+(1+\delta_n)\kappa_V1_{C_V}(a^n_j)\text{.}$$
Since $\lim_{n\to\infty}\delta_n=0$, if we define $ C:=C_V$,
$\lambda:=\frac{1+\lambda_V}{2}$ and
$\kappa:=\kappa_V(1+\sup_{n\in\NN}\delta_n)$, 
there exists $N_0\in\NN$
such that the drift condition in~\eqref{Dn}
holds for all $n\geq N_0$. 
Note that if
we enlarge $C$ and increase $\kappa$,  
the uniform drift condition in~\eqref{Dn}
remains valid for all $n$ it was valid for before the modification.
Finally, if $N_0>1$, we enlarge $C$ 
by all the representatives of the allotments 
$\XX_1,\ldots,\XX_{N_0}$ (finitely many points) and increase $\kappa$ sufficiently,
so that~\eqref{Dn} also holds for
all $n\in\{1,\ldots,N_0-1\}$.

\noindent \textbf{(b)} Recall that by Definition~\ref{allotmentdef}\textbf{(c)},
the sequence $(r_n:=\text{rad}(\XX_n,V))_{n\in\NN}$ tends to infinity, though perhaps not
monotonically.  
Let $D$  be an open ball of radius $r_D>2\sup_{n\in\NN}\delta_n$ in $\RR^d$.  
Since $D$ is a bounded set, by the definition of $V$-radius (see~\eqref{radiusdef}) and Assumption~\hyperref[A1]{A1}, there exists $n_0\in\NN$
such that $D\subseteq \bigcap_{n\geq n_0}V^{-1}\big([1,r_n)\big)$. 
We now enlarge 
the compact set $C$, constructed in part~\textbf{(a)} of this proof, to contain
the bounded set
\begin{equation}\label{uglyset}
\big(\bigcup_{n< n_0}\RR^d\setminus J^n_0\big)
\cup
\bigcap_{n\geq n_0}V^{-1}\big([1,r_n)\big). 
\end{equation}
We may assume the set $C$ is still compact, since the set in~\eqref{uglyset} is bounded,
and hence the uniform drift condition in~\eqref{Dn} still holds. 

Define a measure $\nu$ on the Borel $\sigma$-algebra of $\RR^d$ by
$\nu(B):=\frac{\Leb(B\cap C)}{\Leb(C)}$ for any measurable set $B$.
For each $n\in\NN$, define a measure on the set of representatives $X_n$ 
by $\nu_n(\{a^n_j\}):=\nu(J^n_j)$.
Define the constant $\gamma:=\Leb(C)\inf_{y,x\in C\times C}\alpha(x,y)q(x,y)$
and note that it is strictly positive by 
Assumptions~\hyperref[A2]{A2} and~\hyperref[A3]{A3}
and Definition~\ref{allotmentdef}\textbf{(a)}.
For every $n\in\NN$ and every $0\leq i,j\leq m_n$, such that $a^n_i\in C_n$, 
the form of the kernel $P$ in~\eqref{kernelform}
implies the minorisation condition in~\eqref{uniformminor}:
$$(p_n)_{ij}= P(a^n_i,J^n_j)\geq\int_{J^n_j\cap C} \alpha(a^n_i,y)q(a^n_i,y)dy\geq
\gamma\nu(J^n_j)=\gamma \nu_n\big(\{a_j^n\}\big)\text{.}$$

We now establish the strong aperiodicity condition in~\eqref{gammakaca}. First assume that $n\geq n_0$, let $D'$ be an open ball
of radius $\frac{r_D}{2}$, with the same centre as $D$, and pick $y\in D'$. 
The definition of the $V$-radius $r_n=\text{rad}(\XX_n,V)$ in~\eqref{radiusdef}
implies
$D\cap J^n_0\subseteq  V^{-1}\big([1,r_n)\big)\cap V^{-1}\big([r_n,\infty)\big)$ 
and hence 
$D\cap J^n_0=\emptyset$.
Since the radius $r_D$ of the ball $D$ is strictly greater than $2\sup_{n\in\NN}\delta_n$ and the inequality
$|y-a^n(y)|\leq \sup_{n\in\NN}\delta_n$ holds, it follows that $a^n(y)\in D\subseteq C$. 
Hence, by definition~\eqref{aodx}, it holds that $D'\subseteq\cup_{\{j;a^n_j\in C\}} J^n_j$ and
$$\nu_n(C_n)=\nu_n\left( X_n\cap C\right)=\nu\left(\cup_{\{j;a^n_j\in C\}}
J^n_j\right)\geq\nu(D')=\frac{\Leb(D')}{\Leb(C)}>0\text{.}$$
If $n<n_0$, then it holds that
$C_n=X_n\cap C\supset \{a^n_j:j=1,\ldots,m_n\}$, since 
$C$ 
contains the set in~\eqref{uglyset} and hence 
$\RR^d\setminus J_0^n$.
Therefore we find $\nu_n(C_n)\geq \frac{\Leb(\RR\setminus J_0^n)}{\Leb(C)}>0$.
Hence~\eqref{gammakaca}
holds for the positive constant 
$$\tilde{\gamma}:=
\frac{1}{\gamma}\min\left\{\frac{\Leb(D')}{\Leb(C)},\min_{n<n_0} \frac{\Leb(\RR\setminus J_0^n)}{\Leb(C)}\right\}.
$$ 
This concludes the proof of the proposition.
\end{proof}


Proposition~\ref{uniformdrift} allows us to control the convergence to
stationarity of the approximating chains
uniformly in $n\in\NN$.
In the notation of Theorem~\ref{thetheorem} and Remarks~\ref{pXremark} and~\ref{rem:notation} the following statement holds. 

\begin{prop}\label{operatorprop} 
There exist positive constants $\zeta$ and
$\theta<1$, such that 
the inequality 
\begin{eqnarray*}
\sup_{\|g\|_{v_n}\leq 1} \left|(p_n^kg)(b)-\pi_n(g)\right|  \leq   \zeta\theta^kv_n(b)
\qquad\text{holds for all $b\in X_n$, $k\in\NN\cup\{0\}$ and $n\in\NN$,}
\end{eqnarray*}
where 
the $v_n$-norm 
of a function
$g:X_n\to\RR$
is
$||g||_{v_n}:=\sup_{b\in X_n}|g(b)|/v_n(b)$
and 
$\pi_n(g)$ 
denotes the integral (i.e. weighted sum) of $g$ with respect to $\pi_n$.   
\end{prop}

\begin{proof}
Pick an arbitrary $n\in\NN$.
According to Proposition~\ref{uniformdrift}, the transition matrix
$p_n$ satisfies the drift condition in~\eqref{Dn}, the minorisation
condition in~\eqref{uniformminor} and the strong aperiodicity condition~\eqref{gammakaca}
with the constants $\kappa,\lambda, \gamma,\tilde{\gamma}$, 
which are independent of the choice of $n$.
Hence,~\cite[Theorem~1.1]{baxendale} (see also~\cite[Theorem~2.3]{meyn}) 
applied to the transition kernel $p_n$ on the state space $X_n$,
yields
$$\sup_{\|g\|_{v_n}\leq1}\left|(p_n^kg)(a^n_j)-\pi_n(g)\right|\leq \zeta(n) v_n(a^n_j)\theta(n)^k$$
for every $k\in\NN\cup\{0\}$, $a^n_j\in X_n$
and constants $\zeta(n)\in(0,\infty)$ and $\theta(n)\in(0,1)$.
Furthermore,~\cite[Theorem~1.1]{baxendale} implies that the constants 
$\zeta(n), \theta(n)$
are only a (chain independent) function of 
$\kappa,\lambda, \gamma,\tilde{\gamma}$
in Proposition~\ref{uniformdrift} and hence do not depend on $n$. 
This concludes the proof.
\end{proof}

\subsection{Functions that asymptotically solve Poisson's equation \hyperref[PE]{PE($P$,$F$)}}
\label{subsec:ApproximateSolution}
In this section we complete the proof of Theorem~\ref{thetheorem}.
By the Dominated Convergence Theorem (DCT), Proposition~\ref{l2var} 
implies that
$(\tilde F_n)_{n\in\NN}$
asymptotically solves~\hyperref[PE]{PE($P$, $F$)}
if the following conditions hold:
\begin{equation}
\label{eq:the_conditions}
\sup_{n\in\NN}||\Delta(\tilde{F}_n)||_{V}<\infty
\qquad\text{and}\qquad
\lim_{n\to\infty}\Delta(\tilde{F}_n)(x)=0\>\> \text{ for $\pi$-a.e. }x\in\RR^d\text{.}
\end{equation}

The inequality in~\eqref{eq:the_conditions} follows from~\eqref{eq:Delta_n} and Proposition~\ref{kje1} below,
which states that the $V$-norm $\tilde F_n$, shifted by an appropriate constant, 
is bounded uniformly in $n\in\NN$. The existence of these constants 
rests on the uniform convergence to stationarity 
in Proposition~\ref{operatorprop} above.

The limit in~\eqref{eq:the_conditions} is established by bounding 
$|\Delta(\tilde{F}_n)|$
by a sum of three non-negative terms (see Lemma~\ref{lem:bound_on_Delta_n} below)
and controlling each one separately. 
The first,
given by
$|F(x)-F(a^n(x))|$,
tends to zero 
by~\eqref{aconv}
since the force function 
$F$ 
is assumed to be continuous $\pi$-a.e.
The second term
$|U(x)-U(a^n(x))|$, where 
$U:=P\tilde{F}_n-\tilde{F}_n$,
is controlled by
Proposition~\ref{kje1} and the DCT.
Controlling the third term 
$|\pi_n(f_n)-\pi(F)|$
is more involved.
It requires constructing a further approximating chain
(based on the transition kernel $P$)
with state space $X_n$ and a transiont matrix $p_n^*$, 
whose invariant distribution can be described analytically
in terms of the density $\pi$
(see~\eqref{eq:*_chain} below).
Proposition~\ref{kje2}, 
whose proof also depends on the uniform convergence to stationarity 
in Proposition~\ref{operatorprop},
establishes the desired limit.
We now give the details of the outlined proof. 


\begin{prop}\label{kje1}
There exists a constant $\xi>0$ and a sequence of real numbers $(c_n)_{n\in\NN}$, 
such that the following inequality 
holds for all $n\in\NN$:
$$||\tilde{F}_n+c_n||_V\leq \xi.
$$
\end{prop}

\begin{proof}
Pick an arbitrary $n\in\NN$.
Since $F\in\BB_V$ by assumption, its restriction 
$f_n:X_n\to\RR$
satisfies 
$||f_n||_{v_n}\leq ||F||_{V}$ (see Proposition~\ref{operatorprop} for definition of $v_n$-norm).
By Proposition~\ref{operatorprop}, the  function $\bar f_n:X_n\to\RR$,
given by 
$$\bar{f}_n:=
\sum_{k=0}^{\infty}(p^k_nf_n-\pi_n(f_n)),$$ 
is well defined
and satisfies the inequality 
$\left|\left|\bar{f}_n\right|\right|_{v_n}\leq \frac{\zeta}{1-\theta}|| f_n||_{v_n}\leq \frac{\zeta}{1-\theta}||F||_{V}$.
Furthermore,  by~\cite[Thm.~17.4.2]{tweedie}, the function $\bar{f}_n$ solves Poisson's equation~\hyperref[PE]{PE($p_n$,$f_n$)}. 
Since
$\hat{f}_n:X_n\to\RR$, 
in the definition of 
$\tilde{F}_n$, also solves~\hyperref[PE]{PE($p_n$,$f_n$)}, 
by Remark~\ref{pXremark} there exists a constant $c_n\in\RR$ such that 
$\hat{f_n}+c_n=\bar{f_n}$.

Recall that $\tilde{F}_n=\sum_{j=0}^{m_n}\hat{f}_n(a^n_j)1_{J_j^n}$, pick an
arbitrary $x\in\RR^d$ and note that definition~\eqref{aodx} implies  
$\tilde{F}_n(x)=\hat f_n(a^n(x))$.
Hence,  we obtain
\begin{eqnarray*}
\left|\tilde{F}_n(x)+c_n\right| & = & |\bar{f}_n(a^n(x))|  \leq 
\frac{\zeta}{1-\theta}||F||_V v_n(a^n(x))=\frac{\zeta}{1-\theta}||F||_V
V(a^n(x))\\
& \leq & \xi V(x),\qquad\text{where $\xi:=\frac{\zeta}{1-\theta}(1+\sup_{k\in\NN}\delta_k)||F||_V$}
\end{eqnarray*}
and the last inequality follows from~\eqref{Vbound}.
Since both $x\in\RR^d$ and $n\in\NN$ were arbitrary, this implies  
the proposition.
\end{proof}

In order to analyse the behaviour of the limit in~\eqref{eq:the_conditions}, 
we need to define a further approximating Markov chain on $X_n$ with the transition matrix
$p^*_n$ and the invariant measure
$\pi^*_n$, given by
\begin{equation}
\label{eq:*_chain}
(p^*_n)_{ij}:= \int_{J^n_i}\frac{\pi(x)}{\pi(J_i^n)}P(x,J^n_j)dx\quad\text{and}\quad
\pi^*_n\big(\{a^n_j\}\big):=\pi(J_j^n),\quad 
\text{for  $i,j\in\{0,\ldots ,m_n\}$},
\end{equation}
respectively.
Note that
$(p^*_n)_{ij}=\mathbb{P}_\pi[\Phi_1\in J_i^n|\Phi_0\in J_j^n]$,
where $\Phi$ is the Metropolis-Hastings chain we are analysing.
It is clear from the definition in~\eqref{eq:*_chain} that the equality
$\pi^*_n p^*_n=\pi^*_n$ holds.
Furthermore, if we define a function 
$h_n\colon X_n\to\RR$
by 
\begin{equation}
\label{eq:*_chain_2}
h_n(a^n_j):= \int_{J^n_j}\frac{\pi(x)}{\pi(J_j^n)}F(x)dx\quad\text{for $a_j^n\in X_n$, it holds that}\quad \pi^*_n(h_n)=\pi(F).
\end{equation}

\begin{rema}
\label{rem:Duality}
\begin{enumerate}[(i)]
\item Let 
$\mu$
be a signed measure on 
$X_n$
and 
$\|\mu\|_{v_n}:=\sup_{\|g\|_{v_n}\leq1} |\mu(g)|$
its $v_n$-norm, where  
the norm $\|g\|_{v_n}$ was defined in Proposition~\ref{operatorprop}
and
$\mu(g)$ denotes the integral (i.e. weighted sum) of $g:X_n\to\RR$ with respect to 
$\mu$. 
Furthermore, it is natural to define the \textit{dual} normed vector spaces 
$(\BB_{v_n},||\cdot||_{v_n})$ (analogous to $\BB_{V}$ in~\eqref{BV}) and
$(\MM_{v_n},||\cdot||_{v_n})$
of functions on $X_n$ and signed measures 
on $X_n$, respectively.  
Since $X_n$ is finite,
the vector spaces 
$\BB_{v_n}$
and
$\MM_{v_n}$
are isomorphic to 
$\RR^{1+m_n}$.
Furthermore, any linear function
$B:\BB_{v_n}\to\BB_{v_n}$, mapping $g\mapsto Bg$, induces a linear map on the dual
$B^*:\MM_{v_n}\to\MM_{v_n}$, given by $\mu\mapsto B^*\mu:=\mu B$ 
(in this definition we interpret 
$\mu$ as a row vector and $B$ as a matrix).
It is well known that the  operator 
norms coincide $\|B\|_{v_n}=\|B^*\|_{v_n}$.
This fact, 
which holds in a much more general setting (see~\cite[Section~7]{lasserre}),
plays an important role in the proof of Proposition~\ref{kje2}.

\item
The following estimate holds for 
any point $x\in\RR^d$ and all 
$n\in\NN, y\in\RR^d$: 
\begin{equation}\label{novastvar}
\alpha(a^n(x),y)q(a^n(x),y)\leq \frac{q(y,a^n(x))}{\pi(a^n(x))}\pi(y)\leq  \eta_x \pi(y),
\qquad \text{where }
\eta_x:=\frac{\sup_{z,y\in\RR^d} q(z,y)}{\inf_{n\in\NN}\pi(a^n(x))}. 
\end{equation}
By~\eqref{aconv} and \hyperref[A2]{A2} we have 
$0<\inf\{\pi(z):|z-x|\leq \sup_{k\in\NN}\delta_k\}\leq\pi(a^n(x))$, where 
$\delta_k=\delta(\XX_k,V)$ (see Definition~\ref{allotmentdef}), for all sufficiently large $n\in\NN$. Thus, by~\hyperref[A2]{A2} and~\hyperref[A3]{A3}, 
we have $\eta_x\in(0,\infty)$ and the inequalities in~\eqref{novastvar}, which will be used in the proofs of Proposition~\ref{kje2} and Theorem~\ref{thetheorem},
hold. 
\end{enumerate}
\end{rema}

\begin{prop}\label{kje2}
The following inequalities hold
for the measure $\pi_n^*$  defined in~\eqref{eq:*_chain}:
\begin{equation}\label{vinskatrta}
\big|(\pi^*_n-\pi_n)( f_n)\big|
\leq \frac{\zeta ||F||_V}{1-\theta}\left|\left|\pi^*_n-\pi^*_np_n\right|\right|_{v_n},
\end{equation} 
where the constants $\theta\in(0,1)$ and $\zeta>0$ are as in Proposition~\ref{operatorprop},  and 
\begin{eqnarray}\label{2intbound}
\left\|\pi^*_n-\pi^*_np_n\right\|_{v_n} &\leq&
(1+\sup_{k\in\NN}\delta_k)\int_{\RR^d\times\RR^d}  \big(V(y)+V(x)\big)Z_n(x,y)dy\>\pi(x)dx,
\end{eqnarray}
where $Z_n(x,y):=\big|\alpha(a^n(x),y)q(a^n(x),y)-\alpha(x,y)q(x,y)\big|$ for any
$x,y\in\RR^d$ and the function $a^n(\cdot)$ is given in~\eqref{aodx}.
Furthermore, the following limit holds:
$\lim_{n\to\infty}|\pi_n( f_n)-\pi(F)|=0$.
\end{prop}

\begin{proof}
We estimate the difference 
$|\pi_n( f_n)-\pi(F)|$
using the invariant distribution 
$\pi_n^*$
of the chain 
driven by 
$p_n^*$
and the function 
$h_n$, 
defined in~\eqref{eq:*_chain} and~\eqref{eq:*_chain_2}
respectively,
as follows
\begin{eqnarray}
\nonumber
|\pi_n( f_n)-\pi(F)|&=&|\pi_n( f_n)-\pi^*_n( f_n)+\pi^*_n( f_n)-\pi^*_n(h_n)| \\
&\leq& |(\pi_n-\pi^*_n)( f_n)|+|\pi^*_n( f_n-h_n)|.
\label{decomp}
\end{eqnarray}
We will prove that both terms on the right-hand side  converge to zero as $n\to\infty$.
The definitions of 
$\pi_n^*$ and $h_n$
(in~\eqref{eq:*_chain} and~\eqref{eq:*_chain_2} above)
and the function 
$a^n(\cdot)$
(see~\eqref{aodx})
imply that the second term on the right-hand side of~\eqref{decomp}
takes the form
\begin{eqnarray*}
\pi^*_n( f_n-h_n)&=&\sum_{j=0}^{m_n}\pi(J^n_j)\left(F(a^n_j)-\int_{J^n_j}\frac{\pi(x)}{\pi(J_j^n)}F(x)dx\right)
 =  \int_{\RR^d} \big(F(a^n(x))-F(x)\big)\pi(x)dx\text{.}
\end{eqnarray*}
Since $F$ is continuous $\pi$-a.e.,   the integrand converges to zero $\pi$-a.e. by~\eqref{aconv}.
Furthermore, for any $x\in\RR^d$  it holds that 
\begin{eqnarray*}
\big|F(a^n(x))-F(x)\big| & \leq & \big|F(a^n(x))\big|+\big|F(x)\big| \leq ||F||_V \left(V(a^n(x))+V(x)\right)\\
& \leq & ||F||_V(2+\sup_{k\in\NN}\delta_k)V(x),
\end{eqnarray*}
where the last inequality follows from~\eqref{Vbound}.
Therefore, by the DCT (recall that by the assumption in~\hyperref[A1]{A1} we have $\pi(V)<\infty$), the second term
in~\eqref{decomp} indeed converges to zero.

Establishing the convergence of the first term on the right-hand side in~\eqref{decomp} is more involved. 
We start by establishing the following representation of the signed measure $\pi_n^*-\pi_n$.

\vspace{5pt}

\noindent \textbf{Claim.} There exists a linear map $B_n:\BB_{v_n}\to\BB_{v_n}$, with the dual $B_n^*:\MM_{v_n}\to\MM_{v_n}$, satisfying
$\pi_n^*-\pi_n = B_n^*\left(\pi^*_n-\pi^*_np_n\right) = \left(\pi^*_n-\pi^*_np_n\right)B_n$ 
and $\|B_n^*\|_{v_n}=\|B_n\|_{v_n}\leq \zeta/(1-\theta)$,
where the constants $\theta\in(0,1)$ and $\zeta>0$ are as in Proposition~\ref{operatorprop} 
(see Remark~\ref{rem:Duality}(I) for the definition of $\BB_{v_n}$ and  $\MM_{v_n}$). 

\vspace{5 pt}

Define a transition matrix
$1\otimes\pi_n$
on the state space $X_n$ by
$(1\otimes\pi_n)_{ij}:=\pi_n(a^n_j)$.
The corresponding chain is a sequence of independent r.v.s. with the law given by
$\pi_n$ (independently of the starting distribution).
The inequality in
Proposition~\ref{operatorprop} can therefore be expressed as 
$\|p_n^k-1\otimes \pi_n\|_{v_n}\leq \zeta \theta^k$,
for all $k\in\NN\cup\{0\}$,
implying that 
$B_n:=\sum_{k=0}^{\infty}\left(p^k_n-1\otimes \pi_n\right)$
is a well defined linear map on the normed space $\BB_{v_n}$, such that  
$\|B_n\|_{v_n}\leq \zeta/(1-\theta)$.
In order to establish the first equality in the Claim above, note that 
$\mu (1\otimes\pi_n)=\pi_n$
for any probability measure 
$\mu\in\MM_{v_n}$
and,
by Remark~\ref{rem:Duality}(I) and Proposition~\ref{operatorprop},
the $\|\cdot\|_{v_n}$-norm of the linear operator
$\mu\mapsto \mu (p_n^k-1\otimes \pi_n)$
on 
$\MM_{v_n}$
is bounded above by 
$\zeta \theta^k$
for all 
$k\in\NN$.
In particular, 
$\lim_{k\to\infty}\pi^*_n p^k_n=\pi_n$ in $v_n$-norm
since
$\|\pi^*_n p_n^k-\pi_n\|_{v_n}=\|\pi^*_n(p_n^k-1\otimes \pi_n)\|_{v_n} \leq \zeta \theta^{k}||\pi^*_n||_{v_n}$
for all 
$k\in\NN$.
Consider the identitiy 
$$\left(\pi^*_n-\pi^*_np_n\right)\sum_{k=0}^{\ell}\left(p^k_n-1\otimes \pi_n\right)= \pi^*_n-\pi^*_np_n^{\ell+1}\qquad \text{for all $\ell\in\NN$,}$$
and note that both sides converge in the appropreate 
$\|\cdot\|_{v_n}$-norms as $\ell\to\infty$. 
In the limit, the left-hand side equals  
$\left(\pi^*_n-\pi^*_np_n\right)B_n$
and the right-hand side is 
$\pi_n^*-\pi_n$.
This concludes the proof of the Claim. 

In order to establish the inequality in~\eqref{vinskatrta}, 
note that $\|f_n\|_{v_n}\leq ||F||_V$ and  Remark~\ref{rem:Duality}(I)
imply
$\big|(\pi^*_n-\pi_n)( f_n)\big| \leq \|F\|_V(\pi^*_n-\pi_n)(f_n/\|f_n\|_{v_n})\leq ||F||_V||\pi^*_n-\pi_n||_{v_n}$.
This inequality and the Claim imply~\eqref{vinskatrta}.

The next task is to prove~\eqref{2intbound}.
Let $g\colon X_n\to \RR$ be a function satisfying $\|g\|_{v_n}\leq 1$. 
Recall that $m_n+1$ is the cardinality of $X_n$ and that the function $a^n(\cdot)$ is defined in~\eqref{aodx}.
We apply the definitinons of the stochastic matrix $p_n^*$ and its 
stationary law $\pi_n^*$, given in~\eqref{eq:*_chain},
to obtain 
\begin{align*}
&\left(\pi^*_n-\pi^*_np_n\right)g=\pi^*_n\left(p^*_n-p_n\right)g=\sum_{j=0}^{m_n}\sum_{i=0}^{m_n}\big[\pi(J^n_i)\left((p^*_n)_{ij}-(p_n)_{ij}\right)\big]g(a^n_j)\nonumber\\
&=\sum_{j=0}^{m_n}\left[\int_{\RR^d} \left(P(x,J_j^n)-P(a^n(x),J^n_j)\right)\pi(x)dx\right]g(a^n_j)\nonumber\\
&=\int_{\RR^d} \left(\int_{\RR^d} g(a^n(y))\big[\alpha(x,y)q(x,y)-\alpha(a^n(x),y)q(a^n(x),y)\big]dy\right)\pi(x)dx\nonumber\\
&\hspace{5pt}+\int_{\RR^d} \left(\int_{\RR^d} g(a^n(x))\big[\alpha(a^n(x),y)q(a^n(x),y)-\alpha(x,y)q(x,y)\big]dy\right)\pi(x)dx\text{,}\nonumber
\end{align*}
where the identity 
$\delta_x(J^n_j)g(a^n_j)=\delta_{a^n(x)}(J^n_j)g(a^n_j)=\delta_{a^n(x)}(J^n_j)g(a^n(x))$,
for any
$x\in\RR^d$ and 
$j\in\{0,\ldots,m_n+1\}$,
implies the final equality. 
Since the function
$g\in\BB_{v_n}$, with
$\|g\|_{v_n}\leq 1$, in the calculation above was arbitrary 
and satisfies
$|g(a^n(x))|\leq V(a^n(x))$
for all 
$x\in\RR^d$,
we find
\begin{equation*}
\left|\left|\pi^*_n-\pi^*_np_n\right|\right|_{v_n}=
\sup_{\|g\|_{v_n}\leq 1}|\left(\pi^*_n-\pi^*_np_n\right)g|
\leq \int_{\RR^d\times \RR^d} \big(V(a^n(y))+V(a^n(x))\big)Z_n(x,y)\pi(x)dydx,
\end{equation*}
which, together with~\eqref{Vbound}, implies~\eqref{2intbound}.

We now apply the DCT to deduce that  the right-hand side in~\eqref{2intbound} converges to zero as $n\to\infty$. 
The definition of $Z_n(x,y)$ in the proposition, the form of the transition kernel $P$
in~\eqref{kernelform}, the drift condition in~\hyperref[A1]{A1} and the inequality in~\eqref{Vbound} 
imply the estimates 
\begin{eqnarray*}
\int_{\RR^d} \big(V(y)+V(x)\big)Z_n(x,y)dy & \leq & PV(x)+PV(a^n(x))+2 V(x)\\ 
& \leq & \big((2+\sup_{k\in\NN}\delta_k)\left(\lambda_V+\kappa_V\right)+2\big)V(x)
\end{eqnarray*}
for all $x\in\RR^d$.
Since, by Assumption~\hyperref[A1]{A1}, we have $\pi(V)<\infty$, 
by the DCT the right-hand side in~\eqref{2intbound} tends to zero (as $n\to\infty$) 
if 
\begin{equation}\label{krap}
\lim_{n\to\infty}\int_{\RR^d}\big( V(y)+V(x)\big)Z_n(x,y)dy=0\hspace{10pt}\text{for all $x\in\RR$.}
\end{equation}

To establish the limit in~\eqref{krap}, pick an arbitrary $x\in\RR^d$ 
and note that for every $y\in\RR$ 
it holds that $\lim_{n\to\infty} Z_n(x,y)=0$ by~\eqref{aconv} and the assumptions in~\hyperref[A2]{A2} and~\hyperref[A3]{A3}.
Hence the integrand in~\eqref{krap} converges to zero point-wise. 
By the estimate in~\eqref{novastvar}, the integrand in~\eqref{krap} is
bounded above by the function
$$y\mapsto \big(V(y)+V(x)\big)\big(\eta_x\pi(y)+\alpha(x,y)q(x,y)\big)$$
which does not depend on $n$ and is $\Leb$-integrable in $y\in\RR^d$. Hence the limit  
in~\eqref{krap} holds by the DTC and, consequently, the right-hand side in~\eqref{2intbound} converges to zero
as $n\to\infty$.
This fact and the estimates in~\eqref{vinskatrta} and~\eqref{2intbound}
imply that the first term on right-hand side of~\eqref{decomp}
tends to zero as $n\to\infty$ and the proposition follows. 
\end{proof}

In order to prove  that the limit
$\lim_{n\to\infty}\Delta(\tilde{F}_n)=0$ 
holds $\pi$-a.e. 
(i.e. the second condition in~\eqref{eq:the_conditions}), 
we need the following elementary estimate. 

\begin{lem}
\label{lem:bound_on_Delta_n}
The function 
$\Delta(\tilde{F}_n):\RR^d\to\RR$, can be bounded above as follows:
\begin{eqnarray*}
\left|\Delta(\tilde{F}_n)(x)\right|&\leq& \left|F(x)-F(a^n(x))\right|+|\pi_n( f_n)-\pi(F)| \\
&+&\left|\left(P\tilde{F}_n-\tilde{F}_n\right)(x) -\left(P\tilde{F}_n-\tilde{F}_n\right)(a^n(x))\right|\qquad\text{for all $x\in\RR^d$.}\nonumber
\end{eqnarray*}
\end{lem}

\begin{proof}
Recall that $\tilde{F}_n(x)=\sum_{j=0}^{m_n}\hat{f}_n(a^n_j)1_{J^n_j}(x)$. Hence also 
$P\tilde{F}_n(x)=\sum_{j=0}^{m_n}\hat{f}_n(a^n_j)P(x,J_j^n)$. 
The following equalities hold
\begin{equation}\label{transform}
\Delta(\tilde{F}_n)(b)=P(\tilde{F}_n-\hat{F})(b)-(\tilde{F}_n-\hat{F})(b)=\pi_n( f_n)-\pi(F)\qquad \text{for any $b\in X_n$,}
\end{equation}
since $\hat{F}$ (resp. $\hat{f}_n$)
solves the Poisson equation in~\hyperref[PE]{PE($P$,$F$)} (resp.~\hyperref[PE]{PE($p_n$,$f_n$)}). 
Recall that the function $a^n(\cdot)$ is defined in~\eqref{aodx}. 
Using the definition of $\Delta(\tilde{F}_n)$, the equalities in~\eqref{transform} and the fact that 
$\hat{F}$ solves \hyperref[PE]{PE($P$,$F$)}
yields
\begin{eqnarray*}
\Delta(\tilde{F}_n)(x)
& = & \left(\hat{F}-P\hat{F}\right)(x)-\left(\hat{F}-P\hat{F}\right)(a^n(x))+\left(\hat{F}-P\hat{F}\right)(a^n(x))\\
&   & -\left(\tilde{F}_n-P\tilde{F}_n\right)(a^n(x))+\left(\tilde{F}_n-P\tilde{F}_n\right)(a^n(x)) -\left(\tilde{F}_n-P\tilde{F}_n\right)(x)\\
& = & F(x)-F(a^n(x))+\pi_n( f_n)-\pi(F)+\left(P\tilde{F}_n-\tilde{F}_n\right)(x)-\left(P\tilde{F}_n-\tilde{F}_n\right)(a^n(x))
\end{eqnarray*}
for all $x\in\RR^d$.
The triangle inequality implies the lemma. 
\end{proof}

\begin{proof}[\textbf{Proof of Theorem~\ref{thetheorem}:}] 
By Proposition~\ref{l2var}, it is sufficient to verify that the conditions in~\eqref{eq:the_conditions}
hold for the sequence of functions $(\Delta(\tilde{F}_n))_{n\in\NN}$.
By Proposition~\ref{kje1}
there exists  a constant $\xi'$ and a sequence $(c_n)_{n\in\NN}$ such that the following estimate holds
$$\left|\tilde{F}_n(x)+c_n-\hat{F}(x)\right|\leq \xi' V(x)\qquad\text{for all $n\in\NN$ and $x\in\RR^d$.}$$
Note that
we have
$\Delta(\tilde{F}_n)=P(\tilde{F}_{n}+c_n-\hat{F})-(\tilde{F}_n+c_n-\hat{F})$.
The structure of the transition kernel $P$ in~\eqref{kernelform} implies the following bounds 
for all $n\in\NN$ and $x\in\RR^d$:
\begin{align*}
\left|\Delta(\tilde{F}_n)(x)\right|&\leq\int_{\RR^d} \left(\left|\tilde{F}_{n}(y)+c_n-\hat{F}(y)\right|+\left|\tilde{F}_{n}(x)+c_n-\hat{F}(x)\right|\right)\alpha(x,y)q(x,y)dy\\ 
&\leq \int_{\RR^d} \xi'V(y)\alpha(x,y)q(x,y)dy+\xi'V(x)\int_{\RR^d} \alpha(x,y)q(x,y)dy\\
&\leq \xi'(PV(x) +  V(x)) \leq (\xi'+\xi'\lambda_V +\xi' \kappa_V)V(x),
\end{align*}
where the last inequality is a consequence of the drift condition in~\hyperref[A1]{A1}. 
This inequality and the definition of the $V$-norm in~\eqref{BV}
imply that the first condition in~\eqref{eq:the_conditions} is satisfied. 

We now establish the limit in~\eqref{eq:the_conditions}. 
Fix an arbitrary $x\in\RR^d$, such that  $F$ is continuous at $x$.
The first term on the right-hand side of the inequality 
in Lemma~\ref{lem:bound_on_Delta_n} 
therefore converges to zero by~\eqref{aconv}.
The second term, which is independent of $x$,
tends to zero by Proposition~\ref{kje2}.
In order to deal with the third term on the right-hand side of the inequality in 
Lemma~\ref{lem:bound_on_Delta_n}, note that,
by the definition of $\tilde{F}_n$ in Theorem~\ref{thetheorem}, 
it holds that
$\tilde{F}_n(a^n(x))=\tilde{F}_n(x)$ for all $n\in\NN$. 
Consequently, the structure of the transition kernel  
$P$ in~\eqref{kernelform} implies that 
this term equals 
$|\int_{\RR^d}(\tilde{F}_n(y)-\tilde{F}_n(x))\big[\alpha(x,y)q(x,y)-\alpha(a^n(x),y)q(a^n(x),y)\big]dy|$.
The integrand 
converges to zero for every $y\in\RR^d$ by~\eqref{aconv} and Assumptions~\hyperref[A2]{A2}--\hyperref[A3]{A3}. 
Furthermore, by Proposition~\ref{kje1},  we obtain the inequality 
\begin{equation}\label{tretjibound} \left|\tilde{F}_n(y)-\tilde{F}_n(x)\right|=\left|\tilde{F}_n(y)+c_n-\tilde{F}_n(x)-c_n\right|\leq \xi\left(V(y)+V(x)\right)\quad \text{for every $y\in\RR^d$.}
\end{equation}
The inequality in~\eqref{novastvar} yields an upper bound  
\begin{equation}
	\label{eq:bound_of_Z_n}
|\alpha(x,y)q(x,y)-\alpha(a^n(x),y)q(a^n(x),y)|\leq \eta_x\pi(y)+\alpha(x,y)q(x,y) \quad \text{for all $y\in\RR^d$.}
\end{equation}
The product of the right-hand sides in the inequalities~\eqref{tretjibound} and~\eqref{eq:bound_of_Z_n}
is integrable over 
$\RR^d$ 
with respect to $\Leb(dy)$.
Hence, the DCT implies that 
the third term on the right-hand side of the inequality in Lemma~\ref{lem:bound_on_Delta_n}
converges to zero. 
Therefore, 
$\lim_{n\to\infty}\Delta(\tilde{F}_n)(x)=0$ holds for all $x\in\RR^d$ at which $F$ is continuous. 
It only remains to note that,
by the assumption on $F$ in Theorem~\ref{thetheorem},
this limit holds $\pi$-a.e.
\end{proof}

\section{The rate of decay of asymptotic variances}
\label{sec:Rate}
Theorem~\ref{thetheorem} states that, under~\hyperref[A1]{A1}-\hyperref[A3]{A3}, 
the asymptotic variance $\sigma^2_n$ in~\hyperref[CLT]{CLT($\Phi$, $ F+P\tilde{F}_n-\tilde{F}_n$)}
converges to zero as $n\to\infty$. This section investigates the
speed of this convergence. We show that, 
under suitable Lipschitz and integrability
conditions, the rate of decay is bounded above by the slower of
the decay rates of the sequences $\pi(V^21_{J_0^n})$ and $\delta_n^2=\delta(\XX_n,V)^2$ (see Remark~\ref{A3remark}(i) and
Equation~\eqref{meshdef}
respectively). 
This result suggests that, when constructing an exhaustive sequence of
allotments (see Definition~\ref{allotmentdef} above)
with respect to the drift function $V$, we can guarantee fastest rate of decay
of the asymptotic variance $\sigma^2_n$ when the growth of the bounded set
$\RR^d\setminus J^n_0$
and the decay of the $V$-mesh of the partition of $\RR^d\setminus J^n_0$ are
balanced appropriately ($\delta_n^2$ and $\pi(V^21_{J_0^n})$ must be comparable in size
as $n\to\infty$).

\begin{thm}\label{rate}
Let the assumptions of Theorem~\ref{thetheorem} be satisfied 
and assume that the conditions 
\begin{eqnarray}
&&\limsup_{n\to\infty}\delta_n^{-2} \int_{\RR^d\setminus J^n_0}\Bigg(\int_{\RR^d}\big(V(x)+V(y)\big)Z_n(x,y)dy\Bigg)^2\pi(x)dx<\infty,\label{limsup1}\\
&&\limsup_{n\to\infty}\delta_n^{-2}\int_{\RR^d\setminus J^n_0}|F(x)-F(a^n(x))|^2\pi(x)dx<\infty\label{limsup2}
\end{eqnarray}
hold, where $Z_n(x,y)$, for $x,y\in\RR^d$,
is defined in Proposition~\ref{kje2}
and the function $a^n(\cdot)$ is given in~\eqref{aodx}. 
Then there exists a constant $C_0>0$ such that 
$$
\sigma^2_n\leq C_0 
\max\{\pi(V^21_{J^n_0}),\delta_n^2\}\qquad\text{for all $n\in\NN$.}
$$
\end{thm}

Theorem~\ref{rate}, proved in Section~\ref{subsec:Proofs_rate} below, holds under general conditions that may be hard to verify in
specific examples as the functions in~\eqref{limsup1}--\eqref{limsup2} depend on 
the drift function $V$, often not available in closed form.
With this in mind we study a broad class of Metropolis-Hastings chains 
with the property that $V$ can be described in terms of the target density 
$\pi$ and conditions~\eqref{limsup1}--\eqref{limsup2} 
can be deduced from certain geometric properties of the level sets of $\pi$ near infinity.
Our approach builds on the results in~\cite{roberts2,jarner}.

Consider the class of Random walk Metropolis chains in $\RR^d$. Put differently, 
the proposal density takes the form $q(x,y)=q^*(y-x)$ for some density
$q^*\colon\RR^d\to\RR$. Assume $q^*$ is continuous, strictly positive and
bounded. Assume also that the target $\pi$ is continuously differentiable,
positive and satisfies: 
\begin{equation}\label{eq:RWM_geom}
\lim_{|x|\to\infty}\frac{x}{|x|}\cdot \nabla(\log
\pi)(x)=-\infty\hspace{10pt}\text{and}\hspace{15pt}\limsup_{|x|\to\infty}\frac{x}{|x|}\cdot
\frac{\nabla\pi(x)}{|\nabla\pi(x)|}<0\text{.} 
\end{equation}
Under these assumptions the kernel $P$
in~\hyperref[kernelform]{($\RWM(q,\pi)$)}
satisfies~\hyperref[A1]{A1}-\hyperref[A3]{A3} with a drift function
$V_{\gamma}:=c_{\gamma}\pi^{-\gamma}$ (where $c_{\gamma}$ is a constant that
ensures $V_{\gamma}>1$) for any $0<\gamma<\frac{1}{2}$ (see~\cite[Thms~4.1 and~4.3]{jarner} and Remark~\ref{A3remark}(iv)).
Then the $V_{\gamma}$-radius (see~\eqref{radiusdef}) equals
$\text{rad}(\XX_n,V_{\gamma})=\inf_{y\in J^n_0}c_{\gamma}\pi^{-\gamma}(y)$ 
and the $V_{\gamma}$-mesh 
$\delta_{\gamma,n}=\delta(\XX_n,V_{\gamma})$,
defined in~\eqref{meshdef}, takes the form
\begin{equation}\label{eq:deltagamma}
\delta_{\gamma,n}=\max\left(\sup_{x\notin J^n_0}|x-a^n(x)|, \sup_{x\in \RR^d}\left(\pi(x)/\pi(a^n(x))\right)^{\gamma}-1\right).
\end{equation}

The main assumptions in Proposition~\ref{vrana} below are:\\
\noindent(i) there exists
a function $K_q:\RR^d\to\RR$ and $\epsilon_q>0$ such that 
\begin{equation}\label{qstvar}
\int_{\RR^d}K_q(z)dz<\infty
\quad\text{and}\quad
|q^*(z)-q^*(\tilde{z})|\leq |z-\tilde{z}|K_q(z) \quad \text{for all $z,\tilde z\in\RR^d$ with $|z-\tilde{z}|<\epsilon_q$;}
\end{equation}
(ii) there exist constants $\beta\in(\frac{1}{2},1)$, $c_{\beta}>0$ and $\epsilon_{\pi}>0$ such that 
\begin{equation}\label{pistvar2}
|\nabla\pi(\tilde{x})|<c_{\beta}\pi(x)^{\beta}\quad\text{for all $x,\tilde x\in\RR^d$ with $|x-\tilde{x}|<\epsilon_{\pi}$.}
\end{equation}

\begin{rema}
Assumption~\eqref{qstvar} is a version of a local Lipschitz condition and holds for many proposals $q^*$ used in practice, e.g. normal densities. 
Assumption~\eqref{pistvar2} and condition~\eqref{eq:RWM_geom} hold for instance when target density $\pi$ is proportional to $e^{-p(x)}$, 
for a polynomial $p$ of degree $k$ with leading order terms $p_k$ satisfying $p_k(x)\to\infty$  as $|x|\to\infty$.
\end{rema}

An application of Theorem~\ref{rate} in this setting yields the following result. 

\begin{prop}\label{vrana} 
Assume that~\eqref{qstvar}--\eqref{pistvar2} hold and
fix $\gamma\in(0,\beta-\frac{1}{2})$. Let $(\XX_n)_{n\in\NN}$ be an exhaustive
sequence of allotments with respect to $V_{\gamma}$
defined above.  Let $F\in\BB_{V_{\gamma}}$ be a continuously differentiable function satisfying the inequality
$|\nabla F(\tilde{x})|<c_F\pi^{\gamma-\frac{1}{2}}(x)$ for all $x,\tilde{x}\in\RR^d$ with
$|x-\tilde{x}|<\epsilon_F$ (for some constants $c_F,\epsilon_F>0$). 
Then there exists a constant $C_\gamma>0$ such that 
the asymptotic variance $\sigma_n^2$ in the~\hyperref[CLT]{CLT($\Phi$, $F+P\tilde{F}_n-\tilde{F}_n$)},
where $\tilde F_n$ is 
constructed by the~\hyperref[alg1]{\textbf{Scheme}} with input $P$, $F$ and $\XX_n$,
satisfies
$$
\sigma_n^2\leq C_\gamma \max\left(\delta^2_{\gamma,n},\int_{J^n_0}\pi^{1-2\gamma}(x)dx\right)\qquad\text{for all $n\in\NN$}.
$$
\end{prop}

\begin{rema}
Any polynomial $F$, and in fact any function whose gradient grows no faster than
a polynomial, satisfies assumptions of Proposition~\ref{vrana} 
for any
$\gamma\in(0,\beta-\frac{1}{2})$.
\end{rema}

\subsection{Proofs}
\label{subsec:Proofs_rate}
\begin{proof}[\textbf{Proof of Theorem~\ref{rate}}]
Proposition~\ref{l2var} implies that there exists a constant $C_1>1$ such that 
$\sigma^2_n\leq C_1\cdot\pi\big(\Delta(\tilde{F}_n)^2\big)$ for every $n\in\NN$.
Thus, $\limsup_{n\uparrow\infty}\sigma^2_n/\pi\big(\Delta(\tilde{F}_n)^2\big)<\infty$. 
Futhermore,
the inequality in~\eqref{eq:the_conditions}
implies that $\limsup_{n\uparrow\infty}\pi\big(\Delta(\tilde{F}_n)^21_{J^n_0}\big)/\pi(V^21_{J^n_0})<\infty$. 

Lemma~\ref{lem:bound_on_Delta_n} yields
$\pi\big(\Delta(\tilde{F}_n)^21_{\RR^d\setminus J^n_0}\big)\leq 3(T_1(n)+T_2(n)+T_3(n))$, 
where
\begin{eqnarray*}
&&T_1(n):=\int_{\RR^d\setminus J^n_0}\left|\left(P\tilde{F}_n-\tilde{F}_n\right)(x) -\left(P\tilde{F}_n-\tilde{F}_n\right)(a^n(x))\right|^2\pi(x)dx\text{,}\\
&& T_2(n):=\int_{\RR^d\setminus J^n_0}|F(x)-F(a^n(x))|^2\pi(x)dx\hspace{15pt}\text{and}\hspace{15pt}T_3(n):=|\pi_n(f_n)-\pi(F)|^2\text{.}\nonumber
\end{eqnarray*}
Assumption~\eqref{limsup2} implies 
$\limsup_{n\uparrow\infty} T_2(n)/\delta_n^{2}<\infty$. 
The form of the kernel
$P$ in~\hyperref[kernelform]{$\RWM(q,\pi)$} and the fact that
$\tilde{F}_n(x)=\tilde{F}_n(a^n(x))$ 
for all $x\in\RR^d$
yield 
$$T_1(n)=\int_{\RR^d\setminus
J^n_0}\left|\int_{\RR^d}\left(\tilde{F}_n(y)-\tilde{F}_n(x)\right)\big[\alpha(x,y)q(x,y)-\alpha(a^n(x),y)q(a^n(x),y)\big]dy\right|^2\pi(x)dx\text{.}$$
The inequality in~\eqref{eq:the_conditions} therefore yields
$$ \limsup_{n\uparrow\infty}T_1(n)/\int_{\RR^d\setminus J^n_0}\Big(\int_{\RR^d}\big(V(x)+V(y)\big)Z_n(x,y)dy\Big)^2\pi(x)dx<\infty.$$
Put differntly we obtain 
$\limsup_{n\uparrow\infty} T_1(n)/\delta_n^{2}<\infty$. 

Note that $T_3(n)=|\pi_n(f_n)-\pi(F)|\leq 2|(\pi_n-\pi^*_n)( f_n)|^2+2|\pi^*_n( f_n-h_n)|^2$ 
(recall~\eqref{eq:*_chain}--\eqref{eq:*_chain_2}).
Since
$\pi^*_n( f_n-h_n)=
\int_{\RR^d}(F(x)-F(a^n(x))\pi(x)dx$, the inequality
$F\leq\|F\|_VV$ and~\eqref{Vbound} hold, we find
\begin{align*}
&|\pi^*_n( f_n-h_n)|^2\leq \int_{\RR^d}|F(x)-F(a^n(x))|^2\pi(x)dx\\
&\leq ||F||^2_V(2+\sup_{n\in\NN}\delta_n)^2\pi(V^21_{J^n_0})+\int_{\RR^d\setminus J^n_0}|F(x)-F(a^n(x))|^2\pi(x)dx\text{.}
\end{align*}
Therefore~\eqref{limsup2} yields 
$\limsup_{n\uparrow\infty}|\pi^*_n( f_n-h_n)|^2/\max(\pi(V^21_{J^n_0}),\delta^2_n)<\infty$.
Similarly, inequalities~\eqref{vinskatrta} and~\eqref{2intbound} in
Proposition~\ref{kje2} imply 
$$\limsup_{n\uparrow\infty}|(\pi_n-\pi^*_n)( f_n)|^2/\int_{\RR^d} \left(\int_{\RR^d} \big(V(y)+V(x)\big)Z_n(x,y)dy\right)^2\pi(x)dx<\infty.$$ 
Again, splitting the integral with respect to $x$ into the parts over $J^n_0$ and
$\RR^d\setminus J^n_0$ and applying~\eqref{limsup1}, \hyperref[A1]{A1} and~\eqref{Vbound} 
yields $\limsup_{n\uparrow\infty}|(\pi_n-\pi^*_n)( f_n)|^2/\max(\pi(V^21_{J^n_0}),\delta^2_n)<\infty$.
Hence
$\limsup_{n\uparrow\infty} T_3(n)/\max(\pi(V^21_{J^n_0}),\delta^2_n)<\infty$. 
This concludes the proof of the theorem.
\end{proof}

\begin{proof}[\textbf{Proof of Proposition~\ref{vrana}}] 
Since $P$, $F$ and $\XX_n$ in Proposition~\ref{vrana}
satisfy the assumptions of Theorem~\ref{thetheorem}, we need only to
establish that conditions \eqref{limsup1} and \eqref{limsup2} in Theorem~\ref{rate}
hold for $V=V_{\gamma}$ and $\delta_n=\delta_{\gamma,n}$, defined just before Proposition~\ref{vrana} above. 
Then, since
$\pi(V_{\gamma}^21_{J^n_0})=c^2_{\gamma}\int_{\RR^d}\pi^{1-2\gamma}(x)dx$,
the proposition will follow by Theorem~\ref{rate}. 

Start by establishing~\eqref{limsup2}. We have $|x-a^n(x)|<\delta_{\gamma,n}$ for every $x\in \RR^d\setminus J_0^n$ by~\eqref{eq:deltagamma}. Consequently,
Lagrange's theorem applied to $F$ along a line segment connecting $x$ and $a^n(x)$
yields a point $\tilde{x}^n$ on this segment such that
\begin{align*}
&\delta_{\gamma,n}^{-2}\int_{\RR^d\setminus J^n_0}|F(x)-F(a^n(x))|^2
\pi(x)dx\leq\int_{\RR^d\setminus
J^n_0}\Bigg(\frac{|F(x)-F(a^n(x))|}{|x-a^n(x)|}\Bigg)^2\pi(x)dx\\ &=
\int_{\RR^d\setminus J^n_0}|\nabla F(\tilde{x}^n)|^2\pi(x)dx\leq
c_F \int_{\RR^d}\pi^{2\gamma-1}(x)\pi(x)dx=c_F\int_{\RR^d}\pi^{2\gamma}(x)dx
\end{align*}
holds for a sufficiently large $n$ by assumptions on $F$.
Target $\pi$ decays supper-exponentially along any ray from the origin and 
so does $\pi^{2\gamma}$. Thus, the integral $\int_{\RR^d}\pi^{2\gamma}(x)dx$ is
finite  and~\eqref{limsup2}
follows.

Next, we prove that \eqref{limsup1} holds.
In the setting of a symmetric Random walk Metropolis we have
$\alpha(x,y)=\min\left(1,\pi(x)/\pi(y)\right)$. Let
$\mathcal{A}_x:=\{y\in\RR^d\text{; }\pi(x)\leq\pi(y)\}$ and note  that 
$y\in\mathcal{A}_x$ if and only if  $\alpha(x,y)=1$ and $V_{\gamma}(x)\geq V_{\gamma}(y)$.
Recall $Z_n(x,y)=\big|\alpha(x,y)q^*(y-x)-\alpha(a^n(x),y)q^*(y-a^n(x))\big|$
and, for any $\mathcal{B}\subseteq \RR^d$ and $x\in\RR^d$, denote 
$$\mathcal{I}_n(x,\mathcal{B}):=
\delta_{\gamma,n}^{-2}\left(\int_{\mathcal{B}}\big(V_{\gamma}(x)+V_{\gamma}(y)\big)Z_n(x,y) dy\right)^2\text{.}$$
Condition \eqref{limsup1} is equivalent to 
$\limsup_{n\to\infty}\int_{\RR^d\setminus J^n_0}\mathcal{I}_n(x,\RR^d)\pi(x)dx<\infty$. With this in mind,  
we split the integral in 
$\mathcal{I}_n(x,\RR^d)$
into two integrals,
depending on 
which of the disjoint sets
$\mathcal{A}_x$
and 
$\mathcal{A}^c_x$
the point $y$
belongs to (for any $A\subset\RR^d$, $A^c$ denotes $\RR^d\setminus A$).

Note that it holds
$$\mathcal{I}_n(x,\RR^d)\leq 2\mathcal{I}_n(x,\mathcal{A}_x)+2\mathcal{I}_n(x,\mathcal{A}^c _x)\qquad \text{for all $x\in\RR^d$.}$$


For all sufficiently large $n$,
Lagrange's theorem, \eqref{eq:deltagamma} and~\eqref{pistvar2} imply that
\begin{equation}\label{eq:pizadeva}
\frac{|\pi(a^n(x))-\pi(x)|}{\delta_{\gamma,n}}\leq\frac{|\pi(a^n(x))-\pi(x)|}{|x-a^n(x)|}\leq|\nabla\pi(\tilde{x}^n)|\leq c_{\beta}\pi^{\beta}(x)\qquad\text{for all $x\in\RR^d\setminus J_0^n$.}
\end{equation}

The following holds for all $x,y\in\RR^d$:
\begin{equation}\label{eq:seenaocena}
Z_n(x,y)\leq \alpha(a^n(x),y)\big|q^*(y-a^n(x))-q^*(y-x)\big|+q^*(y-x)\big|\alpha(x,y)-\alpha(a^n(x),y)\big|.
\end{equation}

If $y\in \mathcal{A}_x$ and $n$ is large enough, then for every $x\in\RR^d\setminus J^n_0$, using \eqref{qstvar} and \eqref{eq:pizadeva}, the right hand side of \eqref{eq:seenaocena} can be further bounded as follows (note that $\pi(a^n(x))\geq\pi(y)\geq\pi(x)$ is crucial in the analysis of the right term):
\begin{eqnarray}
Z_n(x,y)&\leq&\delta_{\gamma,n}K_q^*(y-x)+q^*(y-x)\frac{|\pi(a^n(x))-\pi(y)|}{\pi(a^n(x))}1_{\{\pi(a^n(x))>\pi(y)\}}(x,y)\nonumber\\
&\leq&\delta_{\gamma,n}K_q^*(y-x)+\delta_{\gamma,n}c_{\beta}q^*(y-x)\pi^{\beta-1}(x)\nonumber.
\end{eqnarray} 
Since the Lebesgue measure is translation invariant, there exists a constant $c_Z>0$ such that for sufficiently large $n\in\NN$ we have
\begin{equation}\label{eq:celamera1}
\delta_{\gamma,n}^{-1} \int_{\mathcal{A}_x}Z_n(x,y)dy<c_Z\pi^{\beta-1}(x)
\qquad\text{for all $x\in\RR^d\setminus J_0^n$.}
\end{equation}
As $y\in\mathcal{A}_x$, we have $V_{\gamma}(x)\geq V_{\gamma}(y)$. Thus, \eqref{eq:celamera1} and $2\beta-2\gamma-1>0$ imply the following:
\begin{eqnarray}\label{eq:ocena1}
\int_{\RR^d\setminus J^n_0}\mathcal{I}_n(x,\mathcal{A}_x)\pi(x)dx&\leq&
\int_{\RR^d\setminus J^n_0}4V_{\gamma}(x)^2c_Z^2\pi^{2\beta-1}(x)dx\nonumber\\
&=&4c_{\gamma}c_Z^2\int_{\RR^d\setminus J^n_0}\pi^{2\beta- 2\gamma-1}(x)dx<\infty\text{.}
\end{eqnarray}

If $y\in\mathcal{A}^c_x$ and $n$ is large enough, then for every $x\in\RR^d\setminus J^n_0$, using \eqref{qstvar} and \eqref{eq:pizadeva}, we differently bound the right hand side of \eqref{eq:seenaocena} as follows:
\begin{eqnarray}\label{eq:ocenaAc}
Z_n(x,y)&\leq&\frac{\pi(y)}{\pi(a^n(x))}\delta_{\gamma,n}K_q^*(y-x)+q^*(y-x)\frac{\pi(y)}{\pi(a^n(x))}\frac{|\pi(a^n(x))-\pi(x)|}{\pi(x)}\nonumber\\
&\leq&\delta_{\gamma,n}\frac{\pi(y)}{\pi(a^n(x))}\left(K_q^*(y-x)+c_{\beta}q^*(y-x)\pi^{\beta-1}(x)\right)\nonumber\\
&\leq&\delta_{\gamma,n}c_{\pi}\frac{\pi(y)}{\pi(x)}\left(K_q^*(y-x)+c_{\beta}q^*(y-x)\pi^{\beta-1}(x)\right),
\end{eqnarray} 
where $c_{\pi}:=(1+\sup_{n\in\NN}\delta_{\gamma,n})^{1/\gamma}$ (note that $\sup_{n\in\NN}\sup_{x\in\RR^d}\frac{\pi(x)}{\pi(a^n(x))}<c_{\pi}$ by \eqref{eq:deltagamma}). 
Hence, similarly to $\eqref{eq:celamera1}$ there exists a constant $c'_Z>0$ such that
\begin{equation}\label{eq:celamera2}
\delta_{\gamma,n}^{-1} \int_{\mathcal{A}^c_x}Z_n(x,y)dy<c'_Z\pi^{\beta-1}(x)
\qquad\text{for all $x\in\RR^d\setminus J_0^n$.}
\end{equation}
Recall that $V_{\gamma}(y)\geq V_{\gamma}(x)$ for $y\in\mathcal{A}^c_x$ 
and apply the Cauchy-Schwarz inequality to obtain for each $x\in\RR^d\setminus J_0^n$ the bound:
\begin{eqnarray}\label{b2bound_one}
\nonumber
\mathcal{I}_n(x,\mathcal{A}^c_x)
&\leq& 4\delta_{\gamma,n}^{-2}
\int_{\mathcal{A}^c_x}Z_n(x,y) dy\cdot\int_{\mathcal{A}^c_x}V_{\gamma}(y)^2 Z_n(x,y)dy\\
&\leq&4c'_Zc_{\pi}\pi^{\beta-1}(x)\int_{\mathcal{A}^c_x}V_{\gamma}(y)^2\frac{\pi^{\beta}(y)}{\pi^{\beta}(x)}\left(c_{\beta}q^*(y-x)\pi^{\beta-1}(y)+K_q(y-x)\right)dy.
\end{eqnarray}
The second inequality follows by~\eqref{eq:ocenaAc}--\eqref{eq:celamera2} and the inequalities $\pi(y)/\pi(x)<1$ and $\pi(y)^{\beta-1}\geq \pi(x)^{\beta-1}$
for
$y\in\mathcal{A}^c_x$  (recall that $\beta\in(1/2,1)$). 
It is clear that if we substitute 
$\mathcal{A}^c_x$ with $\RR^d$
in~\eqref{b2bound_one}, the inequality remains true. 
Hence the Fubini theorem implies 
\begin{eqnarray}\label{b2bound}
&&\int_{\RR^d\setminus J^n_0} \mathcal{I}_n(x,\mathcal{A}^c_x)\pi(x)dx\nonumber\\
&\leq&4c'_Zc_{\pi} \int_{\RR^d}V_{\gamma}(y)^2\left(c_{\beta}\pi(y)^{\beta-1}\int_{\RR^d}q^*(y-x)dx +\int_{\RR^d}K_q(y-x)dx\right)\pi^{\beta}(y)dy\nonumber\\
&\leq&
4c'_Zc_{\pi}c^2_{\gamma}\left(c_{\beta}\int_{\RR^d}\pi^{2\beta-2\gamma-1}(y)dy+\int_{\RR^d}\pi^{\beta-2\gamma}(y)dy\int_{\RR^d}K_q(z)dz\right)<\infty\text{.}
\end{eqnarray}
Account, that $q^*$ is a density and 
note that assumptions $\gamma\in(0,\beta-1/2)$  and $\beta\in(1/2,1)$ imply
both $\beta-2\gamma,2\beta-2\gamma-1\in(0,1)$  making the integrals in~\eqref{b2bound} finite. This together with \eqref{eq:ocena1} implies the inequality $\limsup_{n\to\infty}\int_{\RR^d\setminus J^n_0}\mathcal{I}_n(x,\RR^d)\pi(x)dx<\infty$ and~\eqref{limsup1} follows.
\end{proof}

\section{Applications of the \hyperref[alg1]{\textbf{Scheme}}}
\label{section6}

Any implementation of the \hyperref[alg1]{\textbf{Scheme}} 
has to tackle the following two issues: 
(a) the stochastic matrix $p_{\XX}$ in step (I) of the~\hyperref[alg1]{\textbf{Scheme}}
cannot be computed analytically; (b) once the approximate solution $\tilde{F}_{\XX}$ has been
computed, the function $P\tilde{F}_{\XX}$, and thus the
control variate $P\tilde{F}_{\XX}-\tilde{F}_{\XX}$, 
are again not accessible in closed form. In Section~\ref{subsec:implementation} we present an implementation
of the~\hyperref[alg1]{\textbf{Scheme}}, feasible for general Metropolis-Hastings chains 
that addresses  these issues. In Section~\ref{subsec:Example}
we apply the method to the symmetric Random walk Metropolis chains with stationary distribution given by a double-well potential 
(i.e. a mixture of normals).
The examples below, satisfying our assumptions, 
are chosen because they are well-known to converge very slowly in the case of the classical ergodic estimator. 

Section~\ref{subsec:Example} illustrates two points.  First, Example~\ref{subsec:Example1} empirically confirms the arbitrary reduction of the asymptotic variance of the ergodic 
average in Theorem~\ref{thetheorem} as the partition of the state space is refined sufficiently.  
Furthermore, the numerical results indicate that the rate of convergence to zero of the asymptotic
variance is of the order specified in Theorem~\ref{rate}. 
Second, and perhaps more importantly for future practical applications, 
Example~\ref{subsec:Example2} demonstrates that an
asymptotic variance reduction can be achieved using a  coarse
partition with few states. This suggests that a similar approach of constructing control
variates could be used for reducing the variance of MCMC algorithms in real-world applications
and highlights the need for further research on how to efficiently construct weak approximations
to the chains of interest in higher dimensions.   

\subsection{Implementation}
\label{subsec:implementation}

Construct a partition $\{J_0,\ldots,J_m\}$  with properties:  
(1) the probability $\pi(J_0)$ is small; 
(2) it is easy to sample uniform random points from sets $J_j$ for $j\neq0$. 
Let $a_j\in J_j$, for $j>0$, be arbitrary and choose 
$a_0$ on the boundary of $J_0$.
One may choose
$J_0$ such that 
$\RR^d\setminus J_0$
contains  (most of) the simulated path of the chain. This works well in practice 
but does not guarantee~(1) and makes the partition dependent on the random output. 

Given the allotment 
$(X,\{J_0,\ldots,J_m\})$,
where
$X=\{a_0,\ldots,a_m\}$,
and the Metropolis-Hastings kernel~\eqref{kernelform},
we have the input required to 
construct the matrix 
$p_{\XX}$
(step (I) of the~\hyperref[alg1]{\textbf{Scheme}}). 
As the precise computation of its entries is not feasible in general, we  construct
an estimate 
$\hat p_{\XX}$ 
of $p_{\XX}$ via IID Monte Carlo.
With this in mind, 
let $i(x)$ be the unique index $i\in\{0,\ldots,m\}$,
such that $x\in J_{i(x)}$, 
and define a random function 
$\hat P : \RR^d\times X\to\RR_+$
by the formula
\begin{equation}
\label{eq:hat_P}
\hat{P}(x,a_j):=\begin{cases}
\frac{1}{n_1}\sum_{l_1=1}^{n_1}\Leb(J_j)\alpha(x,Y_{j,x}^{l_1})q(x,Y_{j,x}^{l_1}) &\text{if } j\notin \{0,i(x)\},\\
\frac{1}{n_2} \sum_{l_2=1}^{n_2}1_{J_0}(Z^{l_2}_x)\alpha(x,Z^{l_2}_x) &\text{if } j=0\neq i(x),\\
1-\sum_{k\in\{0,\ldots,m\}\setminus\{j\}} \hat{P}(x,a_k) & \text{if }i(x)=j,
\end{cases}
\end{equation}
where $n_1,n_2\in\NN$, 
random vectors 
$Y_{j,x}^{l_1}$, $l_1=1,\ldots, n_1$, are IID uniform in the set $J_j$
for any 
$j\in\{1,\ldots,m\}$ (subscript $x$ indicates that $Y_{j,x}^{l_1}$ are simulated at the point $x$ but does not influence the distribution) 
and $Z_x^{l_2}$, $l_2=1,\ldots,n_2$,
are IID random vectors, independent of all $Y_{j,x}^{l_1}$ and distributed according to the proposal distribution  
$q(x,z)dz$
in~\eqref{kernelform}.
We construct the matrix $\hat{p}_{\XX}$ with entries $(\hat{p}_{\XX})_{ij}:=\hat P(a_i,a_j)$ and use it in the~\hyperref[alg1]{\textbf{Scheme}} instead of $p_{\XX}$.


Given a function $F:\RR^d\to\RR$, we can execute steps (II)-(III) in the~\hyperref[alg1]{\textbf{Scheme}}.
Constructing the ergodic average estimator  
$S_k(F+P\tilde{F}_{\XX}-\tilde F_{\XX})$
requires the evaluation of 
the function 
$P\tilde{F}_{\XX}$
along the simulated path 
$(\Phi_i)_{i=1,\ldots,k}$
of the Metropolis-Hastings chain. We use the form of $\tilde F_{\XX}$
and the formula in~\eqref{eq:hat_P}
to define
\begin{equation}
\label{eq:hat_P_F}
\hat P\tilde F_{\XX}(x):=\sum_{j=0}^m(\hat{f}_{\XX})_j\hat{P}(x,a_j)
\end{equation}
for any
$x\in\RR^d$, where 
$\hat f_{\XX}$ is the solution of the system in step~(II) of
the~\hyperref[alg1]{\textbf{Scheme}} obtained by solving Poisson's
equation~\hyperref[PE]{PE$(\hat p_{\XX},f_{\XX})$}.
Moreover, the function $\hat P \tilde F_{\XX}$ is used in place of $P\tilde{F}_{\XX}$
along the entire path of the chain. Put differently, to estimate $\pi(F)$, we use a modified ergodic estimator $S_k(F+\hat{P}\tilde{F}_{\XX}-\tilde F_{\XX})$ instead of the original one $S_k(F+P\tilde{F}_{\XX}-\tilde F_{\XX})$. 

This choice of estimator can be justified as follows:
since $Y_{j,\Phi_i}^{l_1}$ and $Z_{\Phi_i}^{l_2}$,
generated at each time step $i$, in the construction of $\hat P\tilde F_{\XX}(\Phi_k)$ 
are independent of the past 
$(\Phi_j)_{j=1,\ldots,i-1}$,
we can construct a Markov chain $\hat \Phi$ with augmented state space $\RR^d\times (J_1)^{n_1}\times \cdots (J_m)^{n_1}\times (\RR^d)^{n_2}$, 
which keeps track of $\Phi_i$ and the auxiliary variables $Y_{j,\Phi_i}^{l_1}$ and $Z_{\Phi_i}^{l_2}$. It is
not hard to see that the chain $\hat \Phi$ has a unique invariant measure $\hat\pi$ 
satisfying $\hat\pi(F+\hat{P}\tilde{F}_{\XX}-\tilde F_{\XX})=\pi(F+P\tilde{F}_{\XX}-\tilde F_{\XX})=\pi(F)$. 
Furthermore, $\hat\Phi$ is
positive Harris recurrent and hence (by~\cite[Theorem~17.1.7]{tweedie})
the SLLN $S_k(F+\hat{P}\tilde{F}_{\XX}-\tilde F_{\XX})\xrightarrow{k\uparrow \infty}\pi(F)$ a.s. holds for any fixed $n_1,n_2\in\NN$. 

\begin{rema}
\label{rem:depen_path}
The estimator $S_k(F+\hat{P}\tilde{F}_{\XX}-\tilde F_{\XX})$ is unbiased in the
following sense: if the chain 
$\hat\Phi$ is started from
stationarity (i.e. $\hat\Phi_0\sim \hat\pi$) we have
$\text{E}_{\hat\pi}\left[S_k(F+\hat{P}\tilde{F}_{\XX}-\tilde F_{\XX})\right]=\pi(F)$ for any $k\in\NN$. This should
be contrasted with the general approach to variance reduction based on the
Poisson equation~\eqref{PE}, where 
the estimator $S_k(F)$ of $\pi(F)$
is essential in constructing a guess for the solution of~\eqref{PE} 
and hence the control variate itself
(see e.g.~\cite{petros} for this approach applied to random scan Gibbs samplers
and~\cite{Poisson_App_Derivative} for sufficiently smooth transition kernels).
The latter approach produces a consistent but biased estimator 
even if the chain is started in stationarity.
\end{rema}

In order to analyse numerically the level of improvement due to our
implementation of the~\hyperref[alg1]{\textbf{Scheme}},
denote
\begin{equation}
\label{eq:Factor}
r_{k,n}(\XX):=
\frac{\sum_{i=1}^n(S_k^i(F)-\pi(F))^2/n}{\sum_{i=1}^n(S_k^i(F+\hat P\tilde F_{\XX} - \tilde F_{\XX})-\pi(F))^2/n},
\end{equation}
where $n$ is the number of simulated paths of the chain (started in stationarity at independent starting points) and $k$ is the length of each path.
The random vectors 
$(S_k^i(F),S_k^i(F+\hat P\tilde F_{\XX} - \tilde F_{\XX}))$,
for $i=1,\ldots,n$,
are IID samples of the pair of ergodic average estimators $(S_k(F),S_k(F+\hat P\tilde F_{\XX} - \tilde F_{\XX}))$
evaluated on the simulated paths.
Put differently, $r_{k,n}$ is the ratio of mean square errors of estimators
$S_k(F)$ and $S_k(F+\hat P\tilde F_{\XX} - \tilde F_{\XX})$, numerically
evaluated on the same random collection of $n$ independent simulated paths
 and will serve as an estimate of the improvement. 

\subsection{Examples}
\label{subsec:Example}
In both examples we use 
the target law 
$\pi:=\rho N(\mu_1,\sigma^2_1)+(1-\rho) N(\mu_2,\sigma^2_2)$,
where $N(\cdot,\cdot)$ is a normal distribution of the appropriate dimension. 

\subsubsection{One dimensional double-well potential}
\label{subsec:Example1}

Let $\mu_1=-3$, $\sigma_1=1$, $\mu_2=4$, $\sigma_2=1/2$, $\rho= 2/5$.
The target density 
$\pi(\cdot)$
is a mixture of two normal densities with the modes at 
$-3$ and $4$ which takes values close to zero in the neighbourhood 
of the origin. 
Let $F(x):=x^3$ be the force function
and let the proposal density $q(x,\cdot)$ be 
$N(x,1)$. The assumptions of Theorem~\ref{thetheorem} are satisfied in this example. However, the estimator
$S_k(F)$
struggles to converge as the chain tends to get
``stuck'' under one of the modes for a long time, sampling values of
$F$ far away from $\pi(F)$.

Let the allotment $\XX_m$ be defined so that $J^m_0:=\RR\setminus(-8,7]$ and $J^m_j$ for $j=1,2,\dots,m$ are intervals of equal length partitioning $(-8,7]$. We take $a^m_j$ for $j>0$ to be the center of the interval $J^m_j$ and we take $a^m_0=-8$. 
We construct $\hat p_{\XX_m}$ by the formula in~\eqref{eq:hat_P} (using $n_1=n_2=1000$) and $\hat P\tilde{F}_{\XX_m}-\tilde{F}_{\XX_m}$ 
by the formulae in~\eqref{eq:hat_P}--\eqref{eq:hat_P_F} (using $n_1=1$, $n_2=10$) and then use \eqref{eq:Factor} to estimate the factor of improvement of the estimator 
$S_k(F+\hat P\tilde F_{\XX_m} - \tilde F_{\XX_m})$ in comparison to the estimator $S_k(F)$.

The table below shows the ratios of improvement $r_{n,k}(\XX_m)$ as the length
of the paths varies from $k=5\cdot 10^3$ to $2\cdot 10^5$ and the number of
intervals the set $(-8,7]$ is partitioned into varies from $m=30$ to $m=700$.
Each entry was computed using an independent sample of $n=1000$
independent paths of the chain started in stationarity.

\begin{center}
\begin{tabular}{|l||c|c|c|c|}
\hline
$m\setminus k$ & $k=5\cdot 10^3$ & $k=2\cdot 10^4$ & $k=5\cdot 10^4$ & $k=2\cdot 10^5$\\
\hline
\hline
\hline
$m=30$ & $5.93$ & $8.56$ & $9.37$ & $9.62$\\
\hline
$m=50$ & $18.0$ & $32.1$ & $34.2$ & $34.7$\\
\hline
$m=70$ & $39.1$ & $75.5$ & $96.8$ & $97.1$\\
\hline
$m=100$ & $76.9$ & $1.76\cdot 10^2$ & $2.22\cdot 10^2$ & $2.40\cdot 10^2$\\
\hline
$m=300$ & $6.96\cdot 10^2$ & $1.75\cdot 10^3$ & $2.13\cdot 10^3$ & $2.36\cdot 10^3$\\
\hline
$m=500$ & $2.14\cdot 10^3$ & $4.64 \cdot 10^3$ & $6.05\cdot 10^3$ & $6.92\cdot 10^3$\\
\hline
$m=700$ & $3.77 \cdot 10^3$ & $8.90\cdot 10^3$ & $1.16\cdot 10^4$ & $1.32\cdot 10^4$\\
\hline
\end{tabular}
\end{center}

The numerical results support Theorem~\ref{thetheorem} as they demonstrate that the
algorithm is capable of reducing the asymptotic variance arbitrarily. 
Note that the rate of the decay of the asymptotic variance (as the mesh of the allotment decreases)
in Theorem~\ref{rate} and Proposition~\ref{vrana} appears to coincide 
with the growth of the entries in the columns of the table (as $m$ increases).  
This suggests that the bound in Theorem~\ref{rate} (as a function of the mesh) is asymptotically sharp.


\subsubsection{Two dimensional double-well potential}
\label{subsec:Example2}

Let $\mu_1=(-3,0)$, $\sigma^2_1=I$, $\mu_2=(4,0)$, $\sigma^2_2=1/4\cdot I$,
$\rho= 3/5$ ($I$ is a two dimensional identity matrix). Let the force function
be $F(x,y):=x$ and let the proposal density $q(x,\cdot)$ be 
$N(x,I)$. Again, the assumptions of Theorem~\ref{thetheorem} are satisfied.

To specify the allotment, decompose $B:=(-7,6]\times (-4,4]$ into $6=3\times2$
equally sized rectangles and define them to be $J_1,J_2\dots J_{6}$. Take
$J_0:=\RR^2\setminus B$, $a_0:=(-7,0)$ and $a_j$ to be the center of the box
$J_j$ for $j>0$. Construct $\hat p_{\XX_m}$ by the formula in~\eqref{eq:hat_P}
(using $n_1=n_2=1000$) and $\hat P\tilde{F}_{\XX_m}-\tilde{F}_{\XX_m}$ by the
formulae in~\eqref{eq:hat_P}--\eqref{eq:hat_P_F} (using $n_1=1$, $n_2=10$) and
estimate the factor of improvement $r_{k,n}$ in~\eqref{eq:Factor}. 
We obtain approximately a $10\%$ reduction in variance. More precisely we get
$$
\text{$r_{k,n}=1.09$ (resp. $1.08$) for the path of length $k=2\cdot 10^5$ (resp. $k=5\cdot 10^4$),}
$$
where 
$n=1000$
sample paths were used. 
Moreover, $\pi_{\XX}(f_{\XX})$ is a poor estimator of $\pi(F)$ as  
$\big(\pi_{\XX}(f_{\XX})-\pi(F)\big)^2=1.52$, while 
the mean square error of $S_{2\cdot 10^5}(F+\hat P\tilde{F}_{\XX}-\tilde{F}_{\XX})$ is $0.85$.

This indicates that a very fine discretisation need not be necessary to
achieve variance reduction of MCMC estimators. Analogous implementations,
using for example partitions of the state space based on $F$ and $\pi$, might 
lead to variance reduction in higher dimensional models.



\appendix
\section{Existence of exhaustive allotments}
\label{appendix}

\begin{prop}
Let $W\colon \RR^d\to [1,\infty)$ be a continuous function with bounded sublevel sets, i.e. for every $c\in\RR$ the pre-image $W^{-1}((-\infty,c])$ is bounded. Then an exhaustive sequence of allotments with respect to $W$ exists.
\end{prop}

\begin{proof}
Let $(r_n)_{n\in\NN}$ be an increasing unbounded sequence of positive numbers, such that $r_1>\inf_{x\in\RR^d} W(x)$. For each $n\in\NN$ define sets $L_n:=W^{-1}\big((-\infty,r_n)\big)$,
$$\tilde{L}_n:=\{x\in\RR^d; \exists y\in L_n \text{, such that } |x-y|<\sqrt{d}\}\text{.}$$
Set $\tilde{L}_n$ is bounded and non-empty by definitions of $W$ and $r_n$. 
So, $W$ is uniformly continuous on $\tilde{L}_n$. There exists a positive sequence $(\epsilon_n)_{n\in\NN}$ (satisfying $\lim_{n\to\infty}\epsilon_n=0$ and $\sup_{n\in\NN}\epsilon_n<1$) such that $|x-y|<\epsilon_n \sqrt{d}$ implies $|W(x)-W(y)|<\frac{1}{n}$ for each $n\in\NN$ and all $x,y\in \tilde{L}_n$.

Fix $n\in\NN$. For $x=(x_1,x_2,\dots,x_d)\in\RR^d$ denote $K^n_x:=[x_1,x_1+\epsilon_n)\times\dots\times[x_d,x_d+\epsilon_n)$. Clearly, it is possible to pick $x^1,x^2,\dots x^{m_n}\in\RR^d$ so that sets $K^n_j:=K^n_{x^j}$ (for $1\leq j\leq m_n$) are disjoint and cover $L_n$ (assume the cover is minimal). Finally, take $J^n_0$ to be the closure of $\RR\setminus\bigcup _{j=1}^{m_n}K^n_j$ and define $J^n_j:=K^n_j\setminus J^n_0$.
Note that $\Leb(J_j^n)>0$ for all $0\leq j\leq m_n$. For $1\leq j\leq m_n$ pick arbitrary $a_j^n\in J^n_j$ and choose $a_0\in J^n_0$, so that $W(a_0^n)=\inf_{x\in J^n_0}W(x)$ (possible since $W$ has bounded sublevel sets and $J^n_0$ is closed). Sets $J^n_j$ together with representatives $a_j^n$ define an allotment $\XX_n$.

By Pythagoras theorem $|x-y|<\epsilon_n\sqrt{d}$, for $x,y$ from the same $\in J_j^n$. Since $\epsilon_n<1$ and $K^n_j\cap L_n\neq \emptyset$, we get $J^n_j\subset K^n_j\subset \tilde{L}_n$ for all $1\leq j\leq m_n$. Hence, 
$$\max_{1\leq j\leq m_n}\sup_{y\in J_j^n}|y-a^n_j|\leq\epsilon_n\sqrt{d}$$ 
and by uniform continuity (recall $W\geq 1$)
$$\max_{0\leq j\leq m_n}\sup_{y\in J_j^n}\frac{W(a^n_j)-W(y)}{W(y)}\leq \frac{1}{n}\text{.}$$

Doing the above for every $n\in\NN$ shows $\lim_{n\to\infty} \delta(\XX_n,W)=0$ (by \eqref{meshdef}). By \eqref{radiusdef} and definition of $L_n$, $\text{rad}(\XX_n,W)\geq r_n$ for every $n\in\NN$. So, $\lim_{n\to\infty}\text{rad}(\XX_n,W)=\infty$.
\end{proof}

\bibliographystyle{alpha}
\bibliography{Bibliography}

\end{document}